\documentclass[12pt]{amsart}
\usepackage{amsmath,bbm,amsfonts,dsfont,amssymb,amsxtra,amscd,enumerate,amsthm,yfonts,verbatim}
\usepackage{color,xcolor}
\usepackage{latexsym}
\usepackage{epsfig}
\usepackage{fullpage}
\usepackage{hyperref}
\usepackage{comment}
\usepackage[shortlabels]{enumitem}
\usepackage{esint}
\usepackage[scr=rsfso]{mathalfa}

\usepackage{dsfont} 
\newcommand{\1}{\mathds 1}

\newcommand\les{\lesssim}
\newcommand\ges{\gtrsim}

\newcommand\R{\mathbb{R}}
\newcommand\C{\mathbb{C}}
\newcommand\Z{\mathbb{Z}}
\newcommand\N{\mathbb{N}}

\newcommand{\supp}{\text{supp }}

\newcommand\la{\langle}
\newcommand\ra{\rangle}

\newtheorem{theorem}{Theorem}[section]
\newtheorem{lemma}[theorem]{Lemma}
\newtheorem{proposition}[theorem]{Proposition}

\newtheorem{definition}[theorem]{Definition}

\newtheorem{r1}{Remark}

\numberwithin{equation}{section}

\usepackage{mathtools}
\mathtoolsset{showonlyrefs}

\usepackage{mathtools}
\newcommand{\defeq}{\vcentcolon=}
\newcommand{\eqdef}{=\vcentcolon}

\DeclarePairedDelimiter\abs{\lvert}{\rvert}
\DeclarePairedDelimiter\norm{\lVert}{\rVert}
\DeclarePairedDelimiter\inner{\langle}{\rangle}

\makeatletter
\let\oldabs\abs
\def\abs{\@ifstar{\oldabs}{\oldabs*}}
\let\oldnorm\norm
\def\norm{\@ifstar{\oldnorm}{\oldnorm*}}
\let\oldinner\inner
\def\inner{\@ifstar{\oldinner}{\oldinner*}}
\makeatother

\usepackage[mathlines]{lineno}
\definecolor{linenumbergray}{gray}{0.48}

\begin{document}

\title[Transference theory for the trilinear restriction estimate and an application to the quadratic NLS in 3D]{Transference theory for the trilinear restriction estimate and an application to the quadratic NLS in 3D}

\author[I.~Bejenaru]{Ioan Bejenaru} \address[I.~Bejenaru]{Department
  of Mathematics, University of California, San Diego, La Jolla, CA
  92093-0112 USA} \email{ibejenaru@ucsd.edu}

\author[V.~Borges]{Vitor Borges} \address[V.~Borges]{Department
  of Mathematics, University of California, San Diego, La Jolla, CA
  92093-0112 USA} \email{vborges@ucsd.edu}

\author[T.~Chan]{Tiklung Chan} \address[T.~Chan]{Department
  of Mathematics, University of California, San Diego, La Jolla, CA
  92093-0112 USA} \email{tic017@ucsd.edu}

\begin{abstract} 
In this article we develop a theory for effectively transferring trilinear restriction estimates for the Schr\"odinger equation to functions in a more general resolution space adapted to the Schr\"odinger equation. As an application we provide a global theory for the quadratic NLS in three spatial dimensions in the critical Sobolev space $\dot{H}^{-\frac{1}{2}}$. 
\end{abstract}


\maketitle

\setcounter{tocdepth}{1}
\tableofcontents

\section{Introduction}

In this paper we introduce a new method for solving nonlinear Schrödinger equations which is based on the trilinear restriction estimate. The general analysis of nonlinear dispersive PDEs relies heavily on linear estimates (i.e. Strichartz estimates) and bilinear estimates, in particular the bilinear $L^2$ estimate. These estimates are based on transferring estimates known for free waves, for example linear and multilinear restriction estimates, to more general functions in a resolution space where solutions to the underlying nonlinear PDE can be constructed. Transference of linear Strichartz estimates is by now standard. Transference of the bilinear $L^2$ estimate has been made possible by the development of delicate resolution structures such as the $X^{s,b}$ spaces, the $U^2,V^2$ structures adapted to the underlying dispersive flow, or using adapted frames. A more complete exposition of these spaces and applications is provided in Section \ref{s:GA}. 

The first step for a transference theory for trilinear restriction estimates is to establish the estimate for free waves. Recent developments in harmonic analysis, particularly the introduction of the polynomial partitioning method in the study of the restriction conjecture by Guth \cite{Gu-I,Gu-II}, have provided an opening for understanding the trilinear and higher order multilinear restriction estimates. These arguments typically establish ``broad" versions of the restriction estimate, which are well-suited for applications to restriction theory. However, for applications to nonlinear PDEs one needs the genuine multilinear estimates. In \cite{Be-mrec}, the first author adapted these ideas to prove the $(n-1)$-linear restriction estimate in $\R^n$. 

To keep the exposition simpler and more concrete, we assume that we work in $\R^4=\R^3\times\R$, that is three spatial dimensions and one temporal dimension, and focus on the case of paraboloids/Schrödinger equations; this is also the relevant context for the application to the quadratic NLS. In this setting we have the following result.
 
\begin{theorem}[{\cite[Theorem 1.2]{Be-mrec}}] \label{TSS}
Let $f_1, f_2, f_3$ be supported in sets $\Omega_1, \Omega_2, \Omega_3\subset\R^3$ satisfying the transversality condition
\begin{equation} \label{Nres}
|N(\xi_1) \wedge N(\xi_2) \wedge N(\xi_3)| \ges 1, \quad \forall \xi_i \in \Omega_i, i=1,2,3,
\end{equation}
where $N(\xi)$ is the unit normal to the paraboloid $\tau=\xi^2$ in $\R^4$. Then for any $p>\frac{14}{15}$ the following holds true
\begin{equation} \label{TE-b}
    \norm{e^{-it\Delta} f_1 \cdot e^{-it\Delta} f_2 \cdot e^{-it\Delta} f_3}_{L^p(\R^4)}\lesssim \norm{f_1}_{L^2(\R^3)} \norm{f_2}_{L^2(\R^3)} \norm{f_3}_{L^2(\R^3)}
\end{equation}
\end{theorem}

In this paper we develop a transference theory which allows this trilinear estimate to be deployed in the analysis of nonlinear PDEs. In this context, what is actually needed is a trilinear estimate in $L^1(\R^4)$ but in order to obtain it one relies on the stronger estimate above in $L^p$ with $p<1$. Precisely, we prove a trilinear estimate for the low modulation components of functions in  $V^2_\Delta$.

\begin{theorem}\label{thm trilinear transference}
Let $u, v, w \in V^2_\Delta$ such that $Q_{\les 1} u, Q_{\les 1} v, Q_{\les 1} w$ have localization properties similar to the ones in Theorem \ref{TSS}. Then  the following estimate holds true:
\begin{equation} \label{TV2s}
    \norm{Q_{\les 1}u \cdot Q_{\les 1}v \cdot Q_{\les 1}w}_{L^1(\R^4)}\lesssim \norm{u}_{V^2_\Delta} \norm{v}_{V^2_\Delta} \norm{w}_{V^2_\Delta}.
\end{equation}
\end{theorem}

The operator $Q_{\les 1}$ localizes in modulation, that is $|\tau-\xi^2| \les 1$, and it is defined in Section \ref{subsect:not}. Thus, the theorem provides, in some sense, only a partial result since we do not obtain the full estimate for the term $\norm{ u \cdot v \cdot w}_{L^1(\R^4)}$. However, this is not a problem at all in applications, since, armed with \eqref{TV2s}, one is left with the analysis of the trilinear expression $\norm{ u \cdot v \cdot w}_{L^1(\R^4)}$ under the additional assumption that at least one term has high modulation ($\ges 1$). In this case the linear Strichartz and bilinear $L^2$ estimates become more effective. It is precisely in the interaction  $Q_{\les 1}u \cdot Q_{\les 1}v \cdot \overline{Q_{\les 1}w}$ where the resonances live, and the novelty of our method is in being able to obtain the estimate for this term. 

The space $V^2_\Delta$ is defined in section \ref{s:RS}. It is important that the trilinear term is controlled in terms of $V^2_\Delta$ norms only as this makes the estimate very effective in the construction of solutions. It is rather standard to obtain \eqref{TV2s} using the stronger norms $U^2_\Delta$ instead of $V^2_\Delta$ - see for example \cite{Ca, Ca-note, CaHe} - and we provide a proof of this in the appendix. However, this alone would not be enough in applications to PDE.

Now we come to a rather important point; for technical reasons, we actually need a more general form of the above result, see Theorem \ref{trilml}. While Theorem \ref{TV2s} contains one of the new and powerful transference techniques, it is significantly easier than what we actually need in this paper where (after rescaling arguments) it suffices to estimate the term $\norm{Q_{\les 2k}u_k \cdot Q_{\les 2k}v_0 \cdot Q_{\les 2k}w_0}_{L^1(\R^4)}$, where $u_k$ is localized at frequency $2^k \ll 1$, $v_0, w_0$ are localized at frequency $\approx 1$. The main obstacle in estimating this term is that the analogue of \eqref{Nres} in this set-up has a lower bound of $2^k \ll 1$, which represents a degenerate transversality condition for the trilinear restriction estimate. In order to deal with this, we use a standard parabolic rescaling argument to transform the frequency supports into a set-up with transversality constant $\gtrsim 1$. The price to pay is that one term, while still on the paraboloid, moves into the regime with $|\xi| \gg 1$ and here we now encounter degenerate curvature instead. While the full version of \cite[Theorem 1.2]{Be-mrec} for general hypersurfaces is rather robust in that it assumes a minimal amount of curvature information, it is actually not enough for this concrete application we have in mind, except in the case $2^k \approx 1$. In a companion paper \cite{BeCh} we augment the theory in \cite{Be-mrec} to account for even more degeneracy in the curvature information on the underlying hypersurfaces. In fact, because of these degenerate curvature assumptions, even the nature of \eqref{TE-b} changes, see Theorem \ref{ltri1}. 

As an application of this theoretical framework, we have identified the quadratic NLS in three spatial dimensions to be the simplest nonlinear model where we can effectively deploy our new theory. We consider the following quadratic NLS
\begin{equation} \label{QNLS}
(i \partial_t - \Delta) u = c_1 u^2 + c_2 \bar u^2, \quad u(0,x)=u_0(x), 
\end{equation}
where $u: \R^n \times \R \rightarrow \C$ and $c_1,c_2\in\C$. The natural scaling for this equation is $s_c=\frac{n}{2}-2$. 

\eqref{QNLS} has been studied extensively, with works covering a variety of problems such as local well-posedness, global well-posedness, scattering,and ill-posedness. See for instance \cite{bejenaru_sharp_2006, BeDa, CaWe, colliander_bilinear_2001, KeGuVe, MuToTa, NaTaTs, Tsu, tao_multilinear_2001}, and the references therein. Below we describe some of the results most relevant to our current work.

In dimensions $n \geq 4$, global well-posedness for \eqref{QNLS} in the critical space was established by Cazenave and Weissler in \cite{CaWe}. The theory is essentially based on linear Strichartz estimates. The situation changes substantially once $n=3$. Assuming initial data in $H^s(\R^3)$, there are only local (in time) well-posedness results. Local well-posedness for $s > -\frac12$ was established by Tao in \cite{tao_multilinear_2001} using a generalization of the bilinear $X^{s,b}$ estimates from \cite{colliander_bilinear_2001} to higher dimensions. The major obstacle in developing a global in time theory in Sobolev spaces is the lack of an effective way to deal with resonances which are present in this problem.  

On the other hand, for regular data with sufficient spatial decay, the corresponding free solution has temporal decay $\approx t^{-\frac32}$ which is integrable in time. This suggests that a global theory may be within reach for such initial data. 

In \cite{HaNa}, Hayashi and Naumkin established the first global well-posedness and scattering  result for \eqref{QNLS} for small data with $H^{10}$ regularity and decay of type $\la x \ra^{10} u_0 \in L^2$. In \cite{HaMiNa}, Hayashi, Mizumachi, and Naumkin further refined the result by lowering the regularity and decay requirements, as well as providing additional decay properties of the corresponding solution and, in \cite{Ka} Kawahara further refined a decay property of the solutions. All these works rely on the vector field method. 

In \cite{GeMaSh-1}, Germain, Masmoudi, and Shatah introduced the space-time resonance method precisely in the context of the quadratic NLS model in 3D. In this paper, the authors established global well-posedness and scattering for \eqref{QNLS} with initial data requiring less regularity and decay than the above mentioned works; a precise quantification of the level of regularity and decay imposed on the data is not straightforward, given that their conditions on the initial data are made in a non-standard fashion. 

To the best of our knowledge, as explained above, all the global (in time) results in 3D for \eqref{QNLS} are in the context of data that have high regularity and some spatial decay.

The second main result of this paper is a proof of global well-posedness and scattering for \eqref{QNLS} in the critical Sobolev space.

\begin{theorem}\label{thm gwp}
The equation \eqref{QNLS} with small initial data $u_0 \in \dot H^{-\frac12}(\R^3)$ is globally well-posed, scatters to free solutions, and has Lipschitz dependence on its initial data. The solution is also unique in the resolution space $\dot X^{-\frac12} \subset C_t \dot H^{-\frac12}$. 

Moreover, if the initial data is more regular, say $u_0 \in H^s(\R^3)$ for some $s>0$ and its critical norm $\|u_0\|_{\dot H^{-\frac12}(\R^3)}$ is small enough, then the global solution $u(t)$ preserves the regularity.
\end{theorem}

We remark that among the two components of the nonlinearity in \eqref{QNLS}, $(\bar u)^2$ is significantly easier, given the lack of resonances. This is why the paper focuses mostly on the $u^2$ component. The reader may also notice that we do not include the other quadratic component $|u|^2=u \cdot \bar u$ in our analysis. This nonlinearity is known to be rather difficult to handle. In \cite{tao_multilinear_2001}, Tao proved local well-posedness for initial data in $H^s(\R^3)$ when $s>-\frac{1}{4}$. In \cite{ginibre_almost_1995}, Ginibre and Hayashi proved almost global existence in weighted Sobolev spaces using the vector field method. To the best of our knowledge, the global well-posedness problem for this nonlinearity is still open, even for localized and smooth data using the space-time resonance method. In this context, we should also mention the work \cite{guo_global_2025}, where Guo, Liu, and Song proved global well-posedness and scattering for the fractional Schrödinger equation $(i\partial_t-D^\alpha)u=\abs{u}^2$ where $\alpha\in(1,2)$ using the method of space-time resonances. The main difficulty for this nonlinearity is the high-high-to-low interaction which is difficult to control - see Section \ref{s:RA} for some more details on this.

Finally, we remark that after reading this introduction, the flow of the rest of our paper may seem inverted as the quadratic NLS takes the center stage, while the trilinear restriction estimate and the transference theory may be seen as supporting results. There is a good reason for this. The result in Theorem \ref{TSS} and its transference were the starting point of this work; however when working on the quadratic NLS we discovered the need for a far more general version of \eqref{Nres}  not covered by the work of the first author in \cite{Be-mrec}. This  motivated the work in the companion paper \cite{BeCh}, by the first and third authors, which augments that theory so that it delivers the necessary tools to complete the analysis of the quadratic NLS \eqref{QNLS}. In nonlinear PDEs the practical problem at hand dictates the specific multilinear restriction estimate that is needed along with the transference; in our particular case,  even the formalization of \eqref{TE-b} needs nontrivial adjustments.

We do believe that the strategy developed here should work in other nonlinear problems where resonances are present and the classical methods based on linear and bilinear estimates are not strong enough to close the analysis. We provide a robust method of transferring trilinear restriction estimates, hence the central problem remains that of deriving the key trilinear restriction estimate tailored to the particular setup considered. We think that the work in \cite{Be-mrec} and \cite{BeCh} provide enough guidance on how this can be approached, if not already covered by the theory in those papers. 

\subsection{Notation} \label{subsect:not}
We define $A\les B$, if there is a harmless constant $c>0$ such that $A\leq c B$, and $A\ges B$ iff $B\les A$. Further, we define $A\approx B$ iff both $A\les B$ and $B\les A$. Also, we define $A\ll B$ if the constant $c$ can be chosen such that $c<2^{-10}$. Also, $A\gg B$ iff $B\ll A$.

Similarly, we define $A\preceq B$ iff $2^A\les 2^B$, $A\succeq B$ iff $2^A\ges 2^B$, $A \sim B$ iff $2^A\approx 2^B$, $A\prec B$ iff $2^A\ll 2^B$, $A\succ B$ iff $2^A\gg 2^B$.

The Fourier transform of a Schwartz function $f:\R^n \rightarrow \C$ is defined by
\[
\mathcal{F}(f)(\xi)=\hat f(\xi)= \int e^{-i x \cdot \xi} f(x) dx. 
\]
The inverse Fourier transform is defined by 
\[
\mathcal{F}^{-1}(g)(x)=\check g(x)= \frac1{(2\pi)^n} \int e^{i x \cdot \xi} g(\xi) d\xi.
\]
These definitions are then extended to distributions, in particular to $L^p(\R^n)$ spaces, in the usual manner.

Let $\rho^0\in C^\infty_c(-\frac85,\frac85)$ be a fixed smooth, even, cutoff
satisfying $\rho^0(s)=1$ for $|s|\leq \frac54$ and $0\leq \rho\leq 1$. For
$k \in \Z$ we define $\rho_k: \R^n \to \R$,
$\rho_k(y):=\rho^0(2^{-k}|y|)-\rho^0(2^{-k+1}|y|)$. 
Let $\tilde{\rho}_k=\rho_{k-1}+\rho_k+\rho_{k+1}$ and notice that $\tilde{\rho}_k \cdot \rho_k=\rho_k$. For $k \in \Z$, let $P_k$ be the Fourier multiplier with symbol with respect to $\rho_k$ and  $\tilde P_k$ be the Fourier multiplier with symbol with respect to $\tilde{\rho}_k$; we record the identity $\tilde P_k P_k = P_k$. We say that $f$ is localized at frequency $2^k$ if $\hat f$ is supported in the annulus $\{\xi: 2^{k-1} \leq |\xi| \leq 2^{k+1}\}$. We note that $P_k f$ is localized at frequency $2^k$ and that $\tilde P_k f = f$ for any $f$ that  is localized at frequency $2^k$.  We also define $P_{\leq k}=\sum_{0\leq k'\leq k }P_{k'}$, and $P_{> k}=I-P_{\leq k}$. 

In the analysis of the Schr\"odinger equation it is common to use the modulation localization operator. For $j \in \Z$ we define
\[
\mathcal{F}[Q_{j}f](\tau,\xi)=\rho_j(\tau -  \xi^2)\mathcal{F}f(\tau,\xi),
\]
where we recall that $\rho_j$ is the function localizing $|\tau - \xi^2| \approx 2^j$. Similarly we define the operators $Q_{\leq j}, Q_{\prec j}$ etc.  

For $k \in \Z$, we let $\mathcal{I}_k=\{ [2^{k}l,2^{k}(l+1)]: l \in \Z\}$ is the set of intervals of length $2^{k}$ covering $\R$; these are meant to be time interval throughout this paper. Also, $\mathcal C_k$ is the set of cubes of side-length $2^{k}$ on the physical side covering $\R^n$ - this can be obtained by considering the standard $\Z^n$ lattice in $\R^n$ and rescale it by $2^{k}$. 

Given an interval $I \subset \R, |I| >0$, a bump function  $\varphi$ adapted to $I$ satisfies the following bounds
\[
| ((t-c(I)) |I|^{-1})^\alpha (|I| \partial_t)^\beta \varphi | \les_{\alpha,\beta} 1,
\]
for any $\alpha, \beta \in \N \cup \{0\}$. Here $c(I)$ is the center of $I$, but it can chosen to be any other point in $I$. We also denote $\chi_I$ to be the standard characteristic function of the interval $I$; we will also use the notation $\mathds{1}_I$. Since $\chi_I$ is not smooth, we let $\tilde \chi_I$ be a smooth approximation of $\chi_I$ in the sense that $\tilde \chi_I$ is a bump function adapted to $I$ with two additional properties: $\tilde \chi_I(t)=1, \forall t \in I$, thus $\tilde \chi_I \chi_I = \chi_I$, and it is supported in $2I$, the double of $I$. $2I$ is meant to be the double of $I$ from its center, that is if $I=[c(I)-\frac{|I|}2, c(I)+ \frac{|I|}2]$, the $I=[c(I)-|I|, c(I)+ |I|]$. Accordingly we will have the functions $\chi_{2I}$ and $\tilde \chi_{2I}$. 

Given a cube $q \in \R^n$ we construct in a similar manner the concept of bump function adapted to $q$, the characteristic function $\chi_q$, the smooth approximation $\tilde \chi_q$ (which is a bump function adapted to $q$, compactly supported in $2q$, the double of $q$ from its center, and with $\tilde \chi_q \chi_q =\chi_q$).

We write $\chi_q(D_x)$ and $\chi_I(D_t)$ to denote the Fourier multiplier operators with symbols $\chi_q(\xi)$ and $\chi_I(\tau)$, respectively.

\subsection{AI declaration}
All the ideas and arguments in this paper were conceived and written by the authors. No AI was used in the writing or editing of this manuscript. The authors intend to use Gemini for additional spelling and grammar checks in an updated version of this paper.

\section{Outline of the strategy}\label{s:GA}

In this section, we provide a more complete outline of how our new method fits within the context of the general perturbative analysis of dispersive PDEs, in particular the current iteration techniques based on the linear and bilinear restriction theory.

The standard road map to iterate a semilinear dispersive PDE is as follows. For concreteness, consider a general nonlinear Schr\"odinger equation:
\[
(i \partial_t - \Delta)u = N(u, \nabla u), \quad u(0,x)=u_{0}(x),
\]
where $u: \R \times \R^n \rightarrow \C$. One seeks to identify two spaces $X,Y$ (with $X$ being a Banach space) with the following two properties:

\begin{enumerate}[label=(\roman*)]
    \item Solvability of the inhomogeneous equation
    \begin{equation} \label{LXY}
        (i \partial_t - \Delta)u = f, \quad u(0,x)=u_{0}(x),
    \end{equation}
    holds true as follows
    \[
    \|u\|_{X} \les \|f\|_{Y} + \|u_0\|_{L^2(\R^n)}.
    \]
    \item Control of the nonlinearity
    \begin{equation} \label{Nuvb}
        \|N(u,\nabla u)\|_Y \les \| u \|_X^p,
    \end{equation}
    where $p$ is the ``degree" of the nonlinearity of $N(u)$; variations may appear. 
\end{enumerate}

$X$ is referred to as the resolution space for the nonlinear PDE. We ignore the regularity at which one carries out the theory in order to keep this presentation concise but we note that one typically adds regularity to $X$ either directly or by defining spaces $X_k$ meant to measure solutions localized at frequency $\approx 2^k$ and then summing the $X_k$ information (similar to how Sobolev regularity is recouped from information on the frequency localized components).

If, in addition, the resolution space satisfies $X^* \subset Y$ (in a sense that needs to be made precise) then, for the purpose of \eqref{Nuvb}, it suffices to establish instead
\begin{equation} \label{NXY}
|\la N(u, \nabla u), v \ra| = |\int N(u,\nabla u) \bar v dx dt| \les \| u \|_{X}^p \|v\|_{X}.
\end{equation}
Indeed, this guarantees that $N(u,\nabla u) \in X^* \subset Y$ with the appropriate bounds and \eqref{Nuvb} follows.

Linear and multilinear restriction estimates can provide very powerful tools for establishing estimates of the form \eqref{NXY}, especially in the translation invariant setting. For example, for the quadratic nonlinearity $N(u)=u^2$, \eqref{NXY} becomes
\begin{equation} \label{NXYu2}
|\int u^2 \bar v dx dt| \les \| u \|_{X}^2 \|v\|_{X}.
\end{equation}
From this estimate, the connection between \eqref{QNLS} and the trilinear restriction estimate may start to become clearer. In order to formalize these ideas, we start by recalling the linear and bilinear theory for general dispersive flows before discussing the trilinear restriction estimate in this context.

\subsection{Linear and bilinear restriction theory}

Linear restriction estimates, i.e. Strichartz estimates, play a fundamental role in the analysis of dispersive PDEs. For a comprehensive introduction to the topic of Strichartz estimates and their applications, the books of Cazenave \cite{Caz-book} and Tao \cite{Tao-book} are very good references.

In general, Strichartz estimates provide a very effective way of quantifying the dispersive properties of waves, especially in the translation invariant setting. For example, as described in the introduction, in dimensions $n\geq 4$, linear Strichartz estimates are strong enough to provide a perturbative theory for \eqref{QNLS} able to obtain global solutions for small data.

However, there are important nonlinear models where the Strichartz estimates do not capture all the phenomenology necessary to provide a satisfactory perturbative theory even in the case of small initial data; this is the case when the relative strength of the nonlinearity increases. One of the limitations of the Strichartz estimates comes from the fact that they quantify the dispersive properties of only a single wave.

A major improvement in the analysis comes from studying the interaction of multiple waves. Let us recall the bilinear $L^2_{t,x}$ estimate for the Schr\"odinger equation,
\begin{equation} \label{bilS}
\|e^{-it\Delta} f \cdot e^{-it\Delta} g  \|_{L^2_{t,x}} \les 2^{\frac{(n-1)k_1-k_2}2} \|f\|_{L^2(\R^n)} \|g\|_{L^2(\R^n)},
\end{equation}
where $f$ is localized at frequency $\approx 2^{k_1}$, $g$ is localized at frequency $\approx 2^{k_2}$ and $k_1 \leq k_2$. If $n=1$, \eqref{bilS} requires additional transversality assumptions when $|k_1-k_2| \les 1$, but it holds as stated if $n \geq 2$. In this paper we work in dimensions $n \geq 2$ and all the statements made from here on are meant for $n \geq 2$. 

A similar estimate holds true for the wave equation
\begin{equation} \label{bilW}
\|e^{it|\nabla|} f \cdot e^{it|\nabla|} g  \|_{L^2_{t,x}(\R^{n+1})} \les 2^{\frac{(n-1)k_1}2} \theta^{\frac{n-3}2} \|f\|_{L^2(\R^n)} \|g\|_{L^2(\R^n)},
\end{equation}
where $f,g$ are localized at frequency $2^{k_1}$, respectively $2^{k_2}$ and the angle between their support is $\approx \theta$. 

Similar estimates to \eqref{bilS} and \eqref{bilW} hold for other linear dispersive flows too. For all these estimates, the key idea is that the Fourier transform of any free solution to a dispersive equation with constant coefficients ``lives" on the characteristic surface of the equation. For example, free solutions to the Schr\"odinger equation live on the paraboloid $\{\tau-\abs{\xi}^2=0\}$, free solutions to the wave equation live on the cone $\{\tau\pm\abs{\xi}=0\}$, etc. Formally, for the Schrödinger equation for example, one can write $e^{-it\Delta} f = \int e^{i (x,t) \cdot (\xi,\tau)} \hat f(\xi) \delta_{\phi(\xi,\tau)=0} d \xi d \tau$ where $\phi(\xi,\tau)=\tau-\xi^2$. By studying the bilinear $L^2$ estimates on the Fourier side (through Plancherel), one can use the geometry of these surfaces to understand the spacetime interactions of the waves. We provide an overview of these arguments here but note that they are quite standard and can be found in many places, see for instance Chapter 5 of \cite{KTV}.

Let $\Sigma$ be a hypersurface in $\R^d$ which is a non-degenerated level set of the functions $\phi$. A natural way to define $\delta_{\phi=0}$ is given by
\[
\delta_{\phi=0} = |\nabla \phi|^{-1} d \mathcal{H}^{d-1} |_{\phi=0},
\]
where $\mathcal{H}^{d-1}$ is the $d-1$ Hausdorff measure on $\Sigma$.  We now consider two such hypersurfaces $\Sigma_1$ and $\Sigma_2$. For each $(x,y) \in \Sigma_1 \times \Sigma_2$ we let $\Sigma(x,y)=\{y+\Sigma_1\} \cap (x+\Sigma_2)$ and let $L(x,y)$ be the square root of 
\[
\int_{\Sigma(x,y)} [ |\nabla \phi_1(z-y)|^2 |\nabla \phi_2(z-x)|^2 - \la \nabla \phi_1(z-y), \nabla \phi_2(z-x) \ra^2  ]^{-\frac12} d \mathcal{H}^{d-2};
\]
Then we let
\[
L=\sup_{x \in \Sigma_1, y \in \Sigma_2} L(x,y).
\]
With the above notation, the following holds true, see Theorem 5.1 in \cite{KTV},
\begin{equation} \label{bilG}
\| f_1 \delta_{\phi_1=0} \ast f_2 \delta_{\phi_2=0} \|_{L^2(\R^d)} 
\leq L \| f_1 |\nabla \phi_1|^{-1} \|_{L^2(\Sigma_1)} \| f_2 |\nabla \phi_2|^{-1} \|_{L^2(\Sigma_2)}.
\end{equation}
The relation between \eqref{bilG} and \eqref{bilS}/\eqref{bilW} is formalized via Plancherel: instead of estimating a product in $L^2_{t,x}$ as stated in \eqref{bilS}/\eqref{bilW}, one estimates its Fourier transform (which is a convolution) in $L^2_{\tau,\xi}$. It is an easy exercise to verify that \eqref{bilS} and \eqref{bilW} are particular cases of \eqref{bilG}. 

Let us unpack the quantities that appear above. To keep things simpler we assume that, in the support of $\Sigma_1$ and $\Sigma_2$ we have $\|\nabla \phi_1\|_{L^\infty} \approx 1, \|\nabla \phi_2\|_{L^\infty} \approx 1$; this setup can be obtained via rescaling arguments in general. In this case 
\[
[ |\nabla \phi_1(z-y)|^2 |\nabla \phi_2(z-x)|^2 - \la \nabla \phi_1(z-y), \nabla \phi_2(z-x) \ra^2  ]^{-\frac12} \ges \theta^{-1},
\]
where 
\[
\theta = \inf_{x \in \Sigma_1, y \in \Sigma_2} \angle(N_x \Sigma_1, N_y \Sigma_2),
\]
and $N_x \Sigma_1$ is the normal to $\Sigma_1$ at $x \in \Sigma_1$, and similarly for $N_y \Sigma_2$. Thus, under this setup, \eqref{bilG} can be simplified to
\begin{equation} \label{bilGe}
\| f_1 \delta_{\phi_1=0} \ast f_2 \delta_{\phi_2=0} \|_{L^2(\R^d)} 
\leq \theta^{-\frac12} [\sup_{x \in \Sigma_1, y \in \Sigma_2} \mu_{\mathcal{H}^{d-2}} (\Sigma(x,y))]^\frac12 \| f_1 \|_{L^2(\Sigma_1)} \| f_2 \|_{L^2(\Sigma_2)}.
\end{equation}
The bilinear $L^2$ estimate picks out two important factors:
\begin{itemize}
    \item The transversality between $\Sigma_1$ and $\Sigma_2$, quantified by the minimum angle between normals from each surface; the estimate degenerates as $\theta \rightarrow 0$ via a factor of $\theta^{-\frac12}$

    \item The maximum size (quantified by the $\mathcal{H}^{d-2}$ measure) of the intersection between any two translates of $\Sigma_1$ and $\Sigma_2$
\end{itemize}
The above interpretation can be done without assuming the normalizations $\|\nabla \phi_1\|_{L^\infty} \approx 1$ and $\|\nabla \phi_2\|_{L^\infty} \approx 1$, but then the translation would be obscured by the additional factors of $\|\nabla \phi_1\|_{L^\infty(\Sigma_1)}, \|\nabla \phi_2\|_{L^\infty(\Sigma_2)}$. 

We recall that these bilinear estimates are for free waves, that is for solutions to the underlying linear PDE. In the analysis of nonlinear PDEs, one needs to work with more general objects (than free waves). In particular, to prove estimates of the form \eqref{NXY}, one needs to transfer these free wave estimates to more general resolution spaces - this is why designing effective resolution spaces in which solutions to the nonlinear PDE can be constructed and can also replicate the corresponding bilinear $L^2_{t,x}$ estimate has been a very important topic. In \cite{Be-rs} the first author has reviewed the relevant literature of the key resolution structures used for this purpose; we summarize it here and refer the reader to Section 1 of \cite{Be-rs} for additional details.    

\begin{itemize}

\item \bf The $X^{s,b}$ spaces. \rm These spaces appeared in the works of Beals \cite{Bea} (for the wave equation), Bourgain \cite{Bo-1, Bo-2} (for the Schr\"odinger and KdV equations) and Klainerman and Machedon \cite{KlMa} (for the wave equation) and were the first structures known to be effective in replicating estimates known for bilinear interactions of free waves.

Since the bilinear estimate is based in $L^2_{t,x}$, one can use Plancherel and run explicit computations to achieve the transference.

\item \bf The $U^2, V^2$ spaces adapted to the flow. \rm These spaces were introduced by Tataru in the context of the wave equation and later adapted to other flows, see Tataru and Koch \cite{KT, KT1,KT2},  Bejenaru and Tataru \cite{BeTa-2}, Hadac, Herr, and Koch \cite{HHK}, Herr, Tataru, and Tzvetkov \cite{HeTaTz}, to name just the very first works using them.

Based on the definition of the $U^2$ space, it is easy to lift the bilinear $L^2$ estimate to functions in $U^2$. The space $V^2$ provides access to all non-endpoint Strichartz estimates and this implies a (weaker) bilinear $L^2$ estimate. Further optimizing the interpolation process provides the bilinear $L^2$ estimate with some log-type losses for functions in $V^2$; this is necessary as one cannot use (the better) $U^2$ as a resolution space. We remark that more general bilinear $L^p_t L^q_x$ estimates in $U^2$, in particular with $p,q<2$, were proved by Candy and Herr \cite{CaHe} and Candy \cite{Ca}. They also observed that in certain cases one can use $V^2$ for one of the functions without picking up the aforementioned loss.

\item \bf The spaces based on adapted frames. \rm  In studying the wave maps equation, Tataru circumvented the lack of a standard $L^2_t L^\infty_x$ end-point Strichartz estimate by recovering such estimates in adapted frames: fairly rigid frames (and many of them) are needed to capture estimates of type $L^2_{t_\Theta} L^\infty_{x_\Theta}$, and this is paired with energy estimates of type $L^\infty_{t_\Theta} L^2_{x_\Theta}$ to recover the bilinear $L^2$ estimate in \eqref{bilW}. This approach was effective in the study of wave maps \cite{Ta-2}, and similar constructions were used for Schr\"odinger maps in \cite{BIKT-1} and Klein-Gordon type problems, see \cite{BeHe}. 

The rather intricate construction of these resolution spaces is designed so as to recover the clean bilinear $L^2$ estimate. 
\end{itemize}

\subsection{Trilinear restriction theory}

While linear and bilinear estimates are extremely powerful and have been used to study many problems, they have limitations. For example, as described in the introduction, in dimension $n=3$, even bilinear estimates are not strong enough to provide a perturbative theory for \eqref{QNLS}. The major issue is that the nonlinearity has many resonances and there is no special structure, such as a null form, to ``defeat" them in the translation invariant setup. The main goal of this paper is to establish that deploying the trilinear restriction estimate achieves precisely this objective. To the best of our knowledge, there is no result in the current literature which implements a trilinear (or higher order) version of \eqref{bilS} or \eqref{bilW}, or more generally of \eqref{bilG}.

Our exposition below focuses on the explicit problem at hand, the quadratic NLS. However the key features that makes this problem difficult (large set of resonances and lack of a special structure) appear in other nonlinear problems and the method we describe below is likely to be useful in such problems. 

Assuming for now that we have a resolution space $X$ satisfying $X^\ast\subset Y$, then in order to prove \eqref{NXY}, the first natural question to ask is whether the corresponding trilinear estimate holds true for free waves, that is
\begin{equation} \label{trifw}
|\int e^{-it\Delta} f \cdot e^{-it\Delta} g \cdot \overline{e^{-it\Delta} h} dx dt| \les \|f\|_{L^2} \|g\|_{L^2} \|h\|_{L^2}.
\end{equation}
Directly as stated, this is a simpler problem because, just as in the case of the bilinear estimates, this can be restated in the Fourier variables as follows
\[
\int e^{-it\Delta} f \cdot e^{-it\Delta} g \cdot \overline{e^{-it\Delta} h} dx dt= (\tilde f \delta_{\tau-\xi^2=0} \ast \tilde g \delta_{\tau-\xi^2=0}
\ast \tilde h \delta_{\tau+\xi^2=0})(0);
\]
here $\tilde f$ is the ``lift" of $f$ to the characteristic surface and
$\| |\nabla \phi|^{\frac12} \tilde f\|_{L^2(\Sigma)}^2 \approx  \| f\|_{L^2}$, where $\Sigma$ is the paraboloid $\tau=\xi^2$ and $L^2(\Sigma)$ is constructed using the $\mathcal{H}^3$ measure; $\tilde g$ and $\tilde h$ are constructed in a similar fashion. For simplicity we assume that $f,g,h$ are supported at frequency $\les 1$, in which case $\|\tilde f\|_{L^2} \approx \|f\|_{L^2}$ and so on. Estimates of type
\begin{equation} \label{trint}
\|\int f_1 \delta_{\Sigma_1} \ast f_2 \delta_{\Sigma_2} \ast f_3 \delta_{\Sigma_3} \|_{L^\infty} \les \theta^{-\frac12} \|f_1\|_{L^2} \|f_2\|_{L^2} \|f_3\|_{L^2}, 
\end{equation}
have previously appeared in the literature, see for instance \cite{BeCawr} and \cite{BeHeTa}. Here $\theta$ quantifies the transversality between the supports of the three surfaces, that is
\[
\theta = \inf \{|N_{\zeta_1} \Sigma_1 \wedge N_{\zeta_2} \Sigma_2 \wedge N_{\zeta_3} \Sigma_3|, \zeta_i \in \Sigma_i \}, 
\]
where $N_\zeta \Sigma$ denotes the normal to $\Sigma$ at $\zeta \in \Sigma$. Notice the similarity between \eqref{bilGe} and \eqref{trint} when it comes to the dependence on the $\theta$ quantifying the degree of transversality between the hypersurfaces involved. 

However, in applications such as \eqref{QNLS}, one needs the stronger version of \eqref{trint}
\begin{equation} \label{triL1}
\|\mathcal{E}_1 f_1  \cdot \mathcal{E}_2 f_2 \cdot \mathcal{E}_3 f_3\|_{L^1(\R^4)} \les \theta^{-\frac12} \|f\|_{L^2} \|g\|_{L^2} \|h\|_{L^2}, 
\end{equation}
where we use the notation $\mathcal{E} f = \mathcal{F}^{-1}(f \delta_{\Sigma})$ and $\mathcal{F}$ is the (space-time) Fourier transform.
 
Generically, without assuming any curvature properties for the surfaces $\Sigma_i$, \eqref{triL1} is conjectured to be true in every (space-time) dimension $n \geq 3$ (this is the minimum dimension to have $\theta>0$). In \cite{BeCaTa}, Bennett, Carbery, and Tao proved estimates of this type for $\theta\approx 1$ but with extra subpolynomial losses which make the estimate unusable for the purpose of solving nonlinear PDEs. To get a sense of the difficulty of the sharp endpoint multilinear restriction problem, i.e. the estimate without the loss, one can consider the closely related multilinear Kakeya problem. In \cite{Gu-main}, Guth proved the endpoint $d$-linear Kakeya conjecture in $\R^d$ by using the polynomial method; see also the work of Zhang in \cite{zhang_endpoint_2018} for generalizations of this result.

It may be helpful to explain why \eqref{triL1} and its higher dimensional counterparts are significantly more difficult than the bilinear estimates such as \eqref{bilS}. By Plancherel, the bilinear $L^2$ estimate can be estimated on the Fourier side where the convolution of measures supported in transversal hypersurfaces can be effectively estimated through direct methods leading to \eqref{bilG}. There is no such mechanism for trilinear (or higher order) estimates.  

Based on these considerations, things might look hopeless even for free solutions, let alone the need to transfer the estimates to more general functions in a resolution space. What makes this possible is the observation that the trilinear $L^1$ estimate \eqref{triL1} is ``native" to $\R^3$, while in $\R^4$ (or higher dimensions) curvature can provide some extra room; more precisely, under appropriate transversality and curvature assumptions on the underlying surfaces, the optimal estimate in $\R^4$ is the following
\begin{equation} \label{triLC}
\|\mathcal{E}_1 f_1  \cdot \mathcal{E}_2 f_2 \cdot \mathcal{E}_3 f_3\|_{L^p(\R^4)} \les \|f\|_{L^2} \|g\|_{L^2} \|h\|_{L^2}, 
\end{equation}
for any $p\geq\frac{14}{15}$; note that this estimate is stated here for $\theta \approx 1$. The point is that the estimate is expected to hold true for some $p<1$, and, by interpolating with the trivial $L^\infty$ estimate, one can obtain the desired $L^1$ estimate for free waves without the aforementioned loss.

The estimate as stated in \eqref{triLC} was established by the first author in \cite{Be-mrec}, up to the endpoint, for surfaces satisfying a very general curvature condition. In particular, it allows for the surfaces to have vanishing principal curvatures in some directions, which plays an important role in our current work. However, even the very general formulation available in \cite{Be-mrec} does not provide us with all the necessary tools needed in our analysis and this is why that theory has been augmented in \cite{BeCh} by the first and third authors. 

The next step in the analysis is to transfer this trilinear restriction estimate from free solutions to an appropriate resolution space. A recent paper of the the first author \cite{Be-rs} provides an opening in this direction. There, in the context of the Schr\"odinger equation, a new technique was developed for establishing a transference mechanism of bilinear restriction estimates (available for free solutions) to functions in $V^2_\Delta$ (to be defined in section \ref{s:RS}) with low modulations; most importantly the transference is lossless, unlike what is normally expected. In this paper we deploy a similar argument for transferring trilinear restriction estimates to functions in $V^2_\Delta$ with low modulations. Theorem \ref{thm trilinear transference} is such an example for the Schr\"odinger equation, although for the particular PDE at hand, that is \eqref{QNLS}, we need the more complex version of it provided by Theorem \ref{trilml}. The low modulation interaction captures precisely the problematic resonances, whose structures deliver the necessary geometric structure for deploying the trilinear restriction estimate.

While in this outline we focus on the role of the trilinear restriction estimate, there are also other interactions in \eqref{NXYu2} which need to be controlled. However, armed with this low modulation estimate, we can assume that at least one term in these other interactions has relatively high modulation, which corresponds in some sense to nonresonant interactions. A careful analysis based on linear and bilinear estimates then provides the necessary estimates in these cases.

The grand conclusion of this analysis is that we are able to close the key estimate \eqref{NXYu2} with $X=V^2_\Delta$ (at the level of frequency localized functions). Recalling the standard fact that closing the estimate \eqref{NXYu2} in $V^2_\Delta$ is sufficient for recovering the solution of \eqref{LXY} in the same space $V^2_\Delta$, with the precise formulation being supplied by Theorem \ref{LDUD}, this allows us to close the iteration and prove Theorem \ref{thm gwp}.

\section{The resolution space} \label{s:RS}

In this section we introduce our resolution spaces. We begin with recalling the basic theory of $U^p,V^p$ spaces, followed by their adapted versions $U^p_\Delta,V^p_\Delta$ in the context of the Schr\"odinger flow along with their standard properties.

\subsection{$U^p$ and $V^p$ spaces} 

In this section we provide a brief introduction of the $U^p,V^p$ spaces; the material here can be found in \cite{KT,HHK,KT1,KT2,CaHe} and the interested reader is advised to consult these papers for any additional details; we would suggest \cite{HHK}, \cite{KTV} and \cite{KT2} as the more comprehensive ones.

Below we consider intervals $(a,b)$ with $-\infty \leq a < b \leq +\infty$. A partition $\tau$ of $(a,b)$ is a finite monotone sequence $a < t_1 < ...< t_N =b$.

We work with functions $u,v: (a,b) \rightarrow L^2(\R^n)$; as is standard in this context, throughout this paper we work with function $u,v$ that are ruled, in the sense that they have left and right limits everywhere. We also consider $1 \leq  p < \infty$ below. 

\begin{definition}  A $U^p(a,b)$ atom is a step function
\[
a(t)=\sum_{k=1}^K \mathds{1}_{[t_{k-1},t_{k})}(t) \phi_k , \quad \sum_{k=1}^K\|\phi_k\|_{L^2(\R^n)}^p=1,
\]
$U^p(a,b)$ is the atomic space based on these atoms, that is $u \in U^p(a,b)$ if
\[
u = \sum_{j \in \N} \lambda_j a_j,
\]
where $(\lambda_j) \in l^1$ and $a_j$ are atoms as above; $U^p(a,b)$ is equipped with the standard norm
\[
\|u\|_{U^p(a,b)} = \inf \{ \sum_{j \in \N}  |\lambda_j |:  u = \sum_{j \in \N} \lambda_j a_j \ \mbox{as above} \}. 
\]
\end{definition}

By definition, functions $u \in U^p$ are right continuous and $\lim_{t \rightarrow a+} u(t)=0$ in $L^2(\R^n)$. If we allow atoms to include the end-point $t_1=a$, we denote the new space by $\tilde U^p$; note that in this context we lose $\lim_{t \rightarrow a+} u(t)=0$. This is particularly useful when we involve the $U^p$ structure on subintervals of $(a,b)$. In this context we have the following result (see Lemma B.5 in \cite{KT2}):
\begin{equation} \label{U2int}
\| u_1 + u_2 \|^2_{\tilde U^2(a,b)}  \leq \| u_1 \|^2_{\tilde U^2(a,c)} + \| u_2 \|^2_{\tilde U^2(c,b)}, 
\end{equation}
where $c \in (a,b)$. 

\begin{definition} Given a function $v: (a,b) \rightarrow L^2(\R^n)$ we define
\[
|v|_{V^p} =\sup_{\tau=\{(t_j)_{j=1,N}\}} \left( \sum_j \| v(t_{j+1}) -v(t_j) \|_{L^2(\R^n)}^p \right)^\frac1p, 
\quad  \|v\|_{V^p} = |v|_{V^p} + \|v\|_{L^\infty_t L^2_x}. 
\]
We let $V^p=V^p(a,b)$ is the space of functions  $v: (a,b) \rightarrow L^2(\R^n)$ for which $ \|v\|_{V^p}$ is finite. 

\end{definition}

It is easy to show that if $v \in V^p$, then $v$ has lateral limits everywhere, including the end-points $a,b$ (this includes the case $a=-\infty$ and $b=+\infty$); for reference see \cite{HHK}. In order to avoid jumps in $V^2$ that are not related to the lateral limits, it is customary to work with $V^2_{rc}$, the subspace of $V^p$ of functions that are right-continuous, or  $V^2_{lc}$, the subspace of $V^p$ of functions that are left-continuous. In addition, a normalization at one of the end-points is usually present: in $V^2_{rc}$ one uses $\lim_{t \rightarrow a+} v(t)=0$ (in $L^2(\R^n)$), while in $V^2_{lc}$ one uses $\lim_{t \rightarrow b-} v(t)=0$. Such a normalization allows one to skip the use of $\|v\|_{L^\infty_t L^2_x}$ in the definition of $\|\cdot \|_{V^p}$ and simply use only $|v|_{V^p}$. 

In practice, in the context of solutions to nonlinear PDEs, we will have a priori knowledge of the continuity in time of solutions, while the normalization at $-\infty$ or $+\infty$ is unnecessary and incompatible with prescribing an initial data at a specific time $t_0$, which itself acts as a normalization. However one can reach any desired normalization by subtracting a constant in time function. 

Next we continue with basic properties of these spaces. We record the embeddings
\begin{equation}
U^p \subset V^p \subset U^q, \quad  1 \leq p < q < \infty;
\end{equation}
the first embedding is straightforward for atoms, hence for the whole $U^p$, while the second requires a bit more work, see for instance \cite{HHK}, Proposition 2.5 and Corollary 2.6 for a complete argument. 

As a consequence of Proposition 2.5 in \cite{HHK}, we have that given $v \in V^2$, $q>2$ and $N \in \N$, there exists a decomposition $v=u_N + v-u_N$ such that 
such that
\begin{equation} \label{V2U2q}
(\ln N)^{-1} \|u_N\|_{U^2} + 2^{N(1-\frac2q)} \|v-u_N\|_{U^q} \leq C(q) \|v\|_{V^2}. 
\end{equation}
This is formalized in the body of the proof of Proposition 2.20 in \cite{HHK}. Moreover, any frequency localization properties on the original $v$ is inherited by $u_N$ as well; to be more precise, if $P(D_x)$ is a frequency localization operator and $P(D_x) u(t) = u(t), \forall t$, then $P(D_x) u_N(t) = u_N(t), \forall t$. 
This property is obvious from the construction of $u_N$ in Proposition 2.5 in \cite{HHK}.

It is important to understand how the $U^p$ and $V^p$ structures interact with decompositions which have good orthogonality properties; we use the following result from Proposition 4.3 in \cite{CaHe}.
\begin{proposition} \label{Pref}
Let $M_1,M_2 \in (0,+\infty)$. For $k \in \N$, let $T_k : L^2(\R^n) \rightarrow L^2(\R^n)$ be a linear and bounded operator (acting spatially) such that
\[
M_1^{-1} \| \sum_{k \in \N}  T_k f \|_{L^2} \leq \left( \sum_{k \in \N}  \|T_k f \|^2_{L^2} \right)^\frac12 \leq M_2 \|f\|_{L^2}, \quad \forall f \in L^2.
\]
If $1 \leq p \leq 2$, then for all $u \in U^p$ we have the bound
\begin{equation} \label{U2dec}
\left( \sum_{k \in \N}  \|T_k u \|^2_{U^p} \right)^\frac12 \leq M_2 \|u\|_{U^p}.
\end{equation}
If $ p \geq 2$, then for all $v \in V^p$ we have the bound
\begin{equation} \label{V2dec}
 \| \sum_{k \in \N}  T_k v \|_{V^p}  \leq M_1 \left( \sum_{k \in \N}  \|T_k v \|^2_{V^p} \right)^\frac12. 
\end{equation}
\end{proposition}
In particular the estimate \eqref{V2dec} for $V^2$ is not the one we need; what we need is the estimate \eqref{U2dec} for $V^2$, but this is not known to be true. This is the reason we provide a weaker substitute in the following Lemma.  

\begin{lemma} \label{l2V2}
Assume that $v \in V^2$ is localized at frequency $2^k$ and we decompose
\[
v=\sum_{q \in \mathcal{C}_j} \chi_q(D_x) v
\]
for some $j \leq k$. If $q > 2$, then the following holds true
\begin{equation}
\sum_{q \in \mathcal{C}_k} \|\chi_q(D_x) v\|^2_{U^q} \les \la k-j \ra^2 \|v\|_{V^2}^2. 
\end{equation}

\end{lemma}

\begin{proof}
We set $N \in \N$ such that $2N(1-\frac2q) \approx 3(k-j)$. 
We use the decomposition $v=u_N+v-u_N$ from \eqref{V2U2q}. From \eqref{U2dec}
we have that
\begin{equation} \label{aux1a}
\sum_{q \in \mathcal{C}_k} \|\chi_q(D_x) u_{N}\|^2_{U^2} \les \|u_N\|_{U^2}^2.  
\end{equation}
On the other hand, for each $q$ it follows that
\[
\|\chi_q(D_x) (v-u_N)\|_{U^q} \les \|v-u_N\|_{U^q} \les \|v-u_N\|_{V^2}. 
\]
From this we obtain
\begin{equation} \label{aux1b}
\sum_{q \in \mathcal{C}_j} \|\chi_q(D_x) (v-u_N)\|^2_{U^q} \les 2^{3(k-j)} 
 \|v-u_N\|_{U^q}^2 \les 2^{2N(1-\frac2q)} \|v-u_N\|_{U^q} \les \|v\|_{V^2}^2,
\end{equation}
given that the cardinality of cubes $q \in \mathcal{C}_j$ for which $\chi_q(D_x) (v-u_N)\neq 0$ is $\approx 2^{3(k-j)}$; here we implicitly use that $v$ and $u_N$ are localized at frequency $2^k$.  

The result of the Lemma then follows from \eqref{aux1a}, \eqref{aux1b}, and the inclusion $U^2 \subset U^q$. 

\end{proof}

The (distributional) time derivative of functions in $U^p,V^p$ plays an important role in the theory. We define
\begin{equation} \label{DU}
DU^p =\{ \partial_t u: u \in U^p \}, \quad DV^p =\{ \partial_t v:  v \in V^p \},
\end{equation}
where $\partial_t$ is meant to be a distributional derivative. An important ingredient is the following formal duality (which is not to be read verbatim)
\[
(DU^p)'=V^q, \quad (DV^p)'=U^q, \quad 1 < p,q< \infty, \frac1p+\frac1q=1.  
\]
By now this duality is completely understood, particularly in light of the work of Koch and Tataru in \cite{KT2}. We use the following formalization from \cite{KT2}, Proposition B.17; see also \cite{CaHe}, Theorem 4.4 and 4.6. 

\begin{theorem} \label{DualUV} If $1 < p,q< \infty, \frac1p+\frac1q=1$, the spaces $DU^p$ and $DV^p$ can be characterized as the space of distributions for which the following norms are finite:
\begin{equation} \label{TDU}
\|f\|_{DU^p}=\sup_{ \phi \in C_0^\infty: \| \phi\|_{V^q} \leq 1}  \left| \int \la  f , \phi \ra_{L^2} dt \right|,
\end{equation}
\begin{equation} \label{TDV}
\|f\|_{DV^p}=\sup_{ \phi \in C_0^\infty: \| \phi\|_{U^q} \leq 1}  \left| \int \la  f , \phi \ra_{L^2} dt \right|,
\end{equation}
\end{theorem}
 
\begin{r1}
In practice, the distributions $f$ that we test will have additional properties such as $f \in L^1_{t, loc} L^2_x$, and, as a consequence, the associated $u$ with the property $f = \partial_t u$ comes with a priori continuity properties $u \in C_t L^2_x$. In order to uniquely identify $u$, either in $U^p$ or $V^p$, one has to impose a normalization, such as the standard one in $U^p$, $\lim_{t \rightarrow a+} u(t)=0$, or the standard one in $V^p$, $\lim_{t \rightarrow b-} u(t)=0$. 
\end{r1}

For $s \in \R$ and $1 \leq p,q \leq +\infty$, we let 
\[
\|f\|_{\dot B^s_{p.q}}=\big\|\big(2^{sj}\|P_{j}(\partial_t)
f\|_{L^p(\R: L^2(\R^n))}\big)_{j \in \Z }\big\|_{\ell^q}. 
\]
We are interested in the particular case $p=2$ and record the following embeddings
\begin{equation} \label{BUV}
\dot B^{\frac12}_{2,1} \subset U^2 \subset V^2 \subset \dot B^{\frac12}_{2,\infty}.
\end{equation}
These embeddings are well known: see for instance $(2.5)$ in \cite{KT1} in the context of the Schr\"odinger flow (which is what we need in this paper) or \cite[Proposition 2.14]{HHK} for a more general version (more general $p$). 

We note that the (time) frequency cut-offs are bounded on the $U^p,V^p$ spaces, that is $P_{\geq k}(D_t), P_{< k}(D_t)$ are bounded operators on $U^p,V^p$; this is based on the fact that the kernel of $P_{< k}(D_t)$ is an $L^1_t$ function. For complete arguments see for instance \cite[Proposition 2.18]{HHK} or \cite[Lemma 4.17]{KTV}.

\subsection{$U^p,V^p$ adapted to the Schr\"odinger flow.}

Consider the  linear Schr\"odinger equation 
\begin{equation}
(i \partial_t - \Delta) \psi = 0, \quad  \psi(0,x)= \psi_0(x),
\end{equation}
where $\psi: \R \times \R^n \rightarrow \C$. The solution $\psi(t)=e^{-it\Delta} \psi_0$ satisfies $e^{it \Delta} \psi(t)=\psi_0 = \mathds{1}_{\R}(t) \psi_0(x)$, thus it is a constant in time function, just like an atom in $U^p$ (to be more precise we would have to consider the solution on some interval $[T,+\infty)$ for some $T\in \R$ to obtain an atom with the standard normalization). Such a nice representation is lost when we seek solutions for a nonlinear PDE, but instead a more realistic expectation is that $e^{it \Delta} \psi(t)$ belongs to $U^2$ or $V^2$. This motivates the definition
\[
U^p_{\Delta}= \{ u: e^{it \Delta} u \in U^p \}, \quad V^p_{\Delta}= \{ u: e^{it \Delta} u \in V^p \}.
\]
Let us highlight some of the key properties of these spaces. We note the straightforward property $e^{-it \Delta} (\chi_I \phi) = \chi_I  e^{-it \Delta}  \phi$, thus the atoms in $U^2_\Delta$ are
\[
a(t)=\sum_{k=1}^K \mathds{1}_{[t_{k-1},t_{k})}(t) e^{-it \Delta}  \phi_k , \quad \sum_{k=1}^K\|\phi_k\|_{L^2(\R^n)}^p=1.
\]
We recall the standard Strichartz estimate for the linear equation
\[
\|e^{-it\Delta} f \|_{L^q_t L^r_x} \les \|f\|_{L^2},
\]
for any pair $(q,r)$ satisfying $2 < q \leq \infty$ and $\frac{2}q + \frac{3}r=\frac{3}2$; we do not need the end-point in this paper.  Using the embedding $V^2_\Delta \subset U^q_\Delta$, and the fact that the Strichartz estimate is straightforward for atoms in $U^q_\Delta$, we obtain the following standard result
\begin{equation} \label{V2Str}
\| u \|_{L^q_t L^r_x} \les \|u\|_{V^2_\Delta}. 
\end{equation}
The following bilinear $L^2$ estimate is also standard
\begin{equation} \label{V2bil}
\|u_{k_1} v_{k_2}\|_{L^2_{t,x}} \les \la k_2-k_1 \ra^2 2^{\frac{2k_1-k_2}2} \|u_{k_1}\|_{V^2_\Delta} \|v_{k_1}\|_{V^2_\Delta},
\end{equation}
where $u_{k_1}, v_{k_2}$ are assumed to be localized at frequency $2^{k_1}$ respectively $2^{k_2}$, and $k_1 \leq k_2$. This is obtained by (a straightforward)
transferring of the estimate for free waves to $U^2_\Delta$, that is
\[
\|u_{k_1} v_{k_2}\|_{L^2_{t,x}} \les 2^{\frac{2k_1-k_2}2} \|u_{k_1}\|_{U^2_\Delta} \|v_{k_1}\|_{U^2_\Delta},
\]
by using the weaker estimate (obtained via the Strichartz estimates)
\[
\|u_{k_1} v_{k_2}\|_{L^2_{t,x}} \les \|u_{k_1}\|_{L^4_t L^6_x} \|v_{k_2}\|_{L^4_t L^3_x} \les 2^{\frac{k_1}2} \|u_{k_1}\|_{L^4_t L^3_x} \|v_{k_2}\|_{L^4_t L^3_x} \les 2^{\frac{k_1}2} \|u_{k_1}\|_{U^4_\Delta} \|v_{k_1}\|_{U^4_\Delta}, 
\]
and optimizing an interpolation based on the decomposition in \eqref{V2U2q}; details can be found in \cite[Corollary 2.21]{HHK} - while the estimate here is for a different equation and dimensions, the process is otherwise identical. 

Let us recall the definition of the standard $X^{0,b,q}$ spaces in this context. We recall from Section \ref{subsect:not} the definition of the modulation localization operator. For $j \in \Z$, $
\mathcal{F}[Q_{j}f](\tau,\xi)=\rho_j(\tau -  \xi^2)\mathcal{F}f(\tau,\xi)$, $\rho_j$ is the function localizing $|\tau - \xi^2| \approx 2^j$.

For $1\leq p\leq \infty$, $b \in \R$, we define the following semi-norm
\[
\|f\|_{\dot{X}^{0,b,p}}=\big\|\big(2^{bj}\|Q_{j}
f\|_{L^2}\big)_{j \in \Z }\big\|_{\ell^p}, 
\]
with the standard modification when $p=\infty$. We note that $b$ quantifies the time regularity after "pulling back" the linear flow, that is the time regularity of the object $e^{it \Delta} u$. Indeed
\[
\|f\|_{\dot{X}^{0,b,p}} \approx \|  e^{it \Delta} f \|_{(\dot B^{b}_{2,p})_t L^2_x},
\]
where $(B^{b}_{2,p})_t L^2_x$ stands for the space of functions $ g: \R \times \R^n$ with $\|g(t, \cdot)\|_{L^2_x(\R^n)} \in \dot B^{b}_{2,p}$ (as a function of $t$) and with the associated semi-norm. 

The following embeddings are well known (see for instance $(2.5)$ in \cite{KT1}):
\begin{equation} \label{UVX}
X^{0,\frac12,1} \subset U^2_\Delta \subset V^2_\Delta \subset X^{0,\frac12,\infty};
\end{equation}
these are a direct consequence of \eqref{BUV} and the definitions of these spaces. 
In particular, at a single modulation scale $ j$, the space $ V_{ \Delta }^{ 2 }
$ inherits the high-modulation structure defined in \cite[Lemma 2.6
(ii)]{Be-rs}:
\begin{align}\label{eq:x-v2}
  \left( 
    \sum_{ T\in \mathcal T_{ -k }  } \lVert Q_{ j } u \rVert_{ L^{ \infty }_{ t
    } L_{ x }^{ 2 } (T) } ^2  
  \right)^{ \frac{1}{2}  }
  \lesssim \lVert Q_{ j } u \rVert _{ U_{ \Delta }^{ 2 } } 
  \sim \lVert Q_{ j } u \rVert _{ V_{ \Delta }^{ 2 } },
\end{align} 
where the implicit constant is uniform for $j\succeq 2k$.

In the previous section we noted that $P_{\geq k}(D_t), P_{< k}(D_t)$ are bounded operators on $U^p,V^p$; as a consequence $Q_{\preceq}, Q_{\succeq}$ are bounded operators on  $V_\Delta^2,U_\Delta^2$.

Moreover, we also have
\begin{equation} \label{V2L2}
\| Q_{\succeq k}(D_t) f \|_{L^2(\R \times \R^{n})} \les 2^{-\frac{k}2} \| f \|_{V^2_\Delta},
\end{equation}
which follows from \eqref{UVX} (the inclusion into $\dot X^{\frac12}_{2,\infty}$ part). 

In what follows we rely on the following simple identity
\[
e^{it\Delta} Q_m ( e^{-it \Delta} f )=   P_m (D_t) f;
\]
a similar identity holds true if $Q_m$ is replaced by $Q_{ \preceq m}$ or $Q_{\succeq m}$.

Similarly to \eqref{DU},  we define the space
\begin{equation} \label{DUD}
DU^p_\Delta = \{ (i \partial_t - \Delta) u: u \in U^p_\Delta \}, \quad DV^p_\Delta = \{ (i \partial_t - \Delta) u: u \in V^p_\Delta \}
\end{equation}
with the induced norm; here $\partial_t, \Delta $ are distributional derivatives. We note that since $u = e^{-it\Delta} \tilde u, \tilde u \in U^2$, then $(i \partial_t - \Delta) u = e^{-it\Delta} i \partial_t \tilde u$ and this justifies the use of notation. Next we adapt the statements in Theorem \ref{DualUV}  to the new context.

\begin{theorem} \label{LDUD} If $1 < p,q< \infty, \frac1p+\frac1q=1$ the spaces $DU^p_\Delta$ and $DV^p_\Delta$ can be characterized as the space of distributions for which the following norms are finite:
\begin{equation} 
\|f\|_{DU^p_\Delta}=\sup_{ \phi \in C_0^\infty: \| \phi\|_{V^q_\Delta} \leq 1}  \left| \int \la  f , \phi \ra_{L^2} dt \right|,
\quad  
\|f\|_{DV^p}=\sup_{ \phi \in C_0^\infty: \| \phi\|_{U^q_\Delta} \leq 1}  \left| \int \la  f , \phi \ra_{L^2} dt \right|,
\end{equation}
\end{theorem}

Just as we had a remark after Theorem \ref{DualUV}, we have the corresponding remark in place in the new context. 

\begin{r1}
In practice, the distributions $f$ that we test will have additional properties such as $f \in L^1_{t, loc} L^2_x$, and, as a consequence, the associated $u$ with the property $f = (i\partial_t -\Delta) u$ comes with a priori continuity properties $u \in C_t L^2_x$. In order to uniquely identify $u$, one has to impose a normalization, such as the standard one in $U^p_\Delta$, $\lim_{t \rightarrow a+} u(t)=0$, or the standard one in $V^p_\Delta$, $\lim_{t \rightarrow b-} u(t)=0$, or the one imposed by the initial data $u(t_0)=u_0$.
\end{r1}

\subsection{The resolution space}
The resolution space we use in this paper is a variant of $V^2_\Delta$ which accounts for the correct regularity. 

First a word about regularity. Assume that $-\frac12 \leq s <0$. In this case $\dot H^{s}$ keeps many of the reasonable properties of a space of functions, in particular it can be realized as the closure of Schwartz functions in its norm and it is a Banach space. For $s \geq 0$ we work with the inhomogeneous Sobolev space $H^s$ which has similarly good properties; this has the advantage of avoiding some technical details about $\dot H^s$ once $s \geq \frac{n}2$.

Correspondingly, for $-\frac12 \leq s <0$, we define $\dot X^s$ to be the space of functions $f: \R^3 \times \R \rightarrow \C$ for which the following norm is finite
\[
\|f\|_{\dot X^s} := \left( \sum_{k \in \Z} 2^{2ks} \|P_k f\|_{V^2_\Delta}^2 \right)^\frac12.
\]
Given that $V^2_\Delta$ is a Banach space, it is easy to see that $\dot X^s$ is a Banach space. For $s \geq 0$ we switch and define the inhomogeneous version $X^s$ using the norm
\[
\|f\|_{X^s} := \left( \|P_{\leq 0} f\|_{L^2}^2 + \sum_{k \in \N} 2^{2ks} \|P_k f\|_{V^2_\Delta}^2 \right)^\frac12.
\]

\section{Resonance analysis} \label{s:RA}

We first consider the nonlinearity $N(u)=u^2$ and seek to understand the resonant regime of the trilinear expression
\[
\begin{split}
I & = \int u v \bar w dx dt = (\hat u \ast \hat v \ast \widehat{ \bar w})(0) \\
& = \int \hat u(\xi,\tau) \hat v(\eta, \tilde \tau) \widehat{ \bar w}(-\xi-\eta, -\tau - \tilde \tau) d \xi d \eta d\tau d \tilde \tau \\
& = \int \hat u(\xi,\tau) \hat v(\eta, \tilde \tau) \overline{ \hat w}(\xi+\eta, \tau + \tilde \tau) d \xi d \eta d\tau d \tilde \tau. 
\end{split}
\]
In the beginning we assume that $u,v,w$ are free waves, thus they are supported on the paraboloid $\tau=\xi^2$. The resonant interactions in this expression are as follows: $\tau = \xi^2$, $\tilde \tau = \eta^2$, and $\tau + \tilde \tau = (\xi+\eta)^2$. This corresponds to $\tau + \tilde \tau= \xi^2+\eta^2=(\xi+\eta)^2 - 2 \xi \cdot \eta$, hence $\xi \cdot \eta=0$ gives rise to classical resonances. We take the analysis of the resonance one step further and ask whether the three directions of wave propagation are linearly independent or not, as this plays a crucial role in this paper. The waves propagate in the co-normal direction, thus we ask whether the three normals to the paraboloid $N(\xi)= \frac{(-2\xi,1)}{\sqrt{1+4|\xi|^2}}, N(\eta)= \frac{(-2\eta,1)}{\sqrt{1+4|\eta|^2}},
N(\xi+\eta)= \frac{(-2\xi-2\eta,1)}{\sqrt{1+4|\xi+\eta|^2}}$ are linearly independent or not. Ignoring the denominators for now, we need to compute
\[
\operatorname{vol}((-2\xi,1), (-2\eta,1), (-2\xi -2\eta,1))= \operatorname{vol}((-2\xi,1), (-2\eta,1), (0, 0, 0,1)),
\]
whose square is given by
\[
\det \begin{vmatrix} 
1+4 \xi^2 & 1 & 1 \\ 1 & 1+4\eta^2 & 1  \\ 1 & 1 & 1
\end{vmatrix}= \det \begin{vmatrix} 
4 \xi^2 & 0 & 0 \\ 0 & 4\eta^2 & 0  \\ 1 & 1 & 1
\end{vmatrix} = 16 |\xi|^2 |\eta|^2. 
\]
If we now allow for the case $\xi \cdot \eta \ne 0$, then the above is modified as follows
\[
\det \begin{vmatrix} 
1+4 \xi^2 & 1+ 4 \xi \cdot \eta & 1 \\ 1+4 \xi \cdot \eta & 1+4\eta^2 & 1  \\ 1 & 1 & 1
\end{vmatrix}= \det \begin{vmatrix} 
4 \xi^2 & 4 \xi \cdot \eta  & 0 \\ 4 \xi \cdot \eta & 4\eta^2 & 0  \\ 1 & 1 & 1
\end{vmatrix} = 16 (|\xi|^2 |\eta|^2 - |\xi \cdot \eta|^2). 
\]
The conclusion is the following:
\[
|N(\xi) \wedge N(\eta) \wedge N(\xi+\eta)|^2 = \frac{16 |\xi \wedge \eta|^2}{(1+4|\xi|^2)(1+4|\eta|^2)(1+4|\xi+\eta|^2)}.  
\]
Assume that we are close to the resonant regime in the sense  $|\xi \cdot \eta| \leq \frac12 |\xi| |\eta|$; in that case $|\xi \wedge \eta|^2 \approx |\xi|^2 |\eta^2|$ and therefore
\[
|N(\xi) \wedge N(\eta) \wedge N(\xi+\eta)|^2 \approx \frac{|\xi|^2 |\eta|^2}{(1+4|\xi|^2)(1+4|\eta|^2)(1+4|\xi+\eta|^2)}.  
\]
First, we consider a low-high to high interaction in which $|\xi| \les |\eta| \approx |\xi+\eta|$. Moreover a rescaling argument allows us to focus on the simpler case when $|\eta| \approx 1$. In that case we have
\begin{equation} \label{ndeg}
|N(\xi) \wedge N(\eta) \wedge N(\xi+\eta)| \approx  |\xi|.  
\end{equation}
Property \eqref{ndeg} plays a crucial role in our work. Geometrically \eqref{ndeg} reads as follows: if we consider two resonant frequencies $(\xi,\abs{\xi}^2),(\eta,\abs{\eta}^2)$ and their interaction output $(\xi+\eta, \abs{\xi+\eta}^2)$, then the corresponding wave propagation directions are not coplanar. Geometrically this makes sense as $ |N(\eta) \wedge N(\xi+\eta)| \approx  |\xi|$ and $N(\xi)$ creates an angle $\approx 1$ with the plane spanned by $N(\eta)$ and $ N(\xi+\eta)$. We note that this holds true because we impose the condition $|\xi \cdot \eta| \leq \frac12 |\xi| |\eta|$. Both statements would fail if $\xi$ and $\eta$ were parallel but this would be a non-resonant interaction. The exact quantification of the non co-planarity is very important, since a value on the right-hand side of \eqref{ndeg} that degenerates too much could potentially jeopardize our method. 

It may be useful to compare \eqref{ndeg} with the similar geometric condition used in the space-time resonance method from \cite{GeMaSh-1}. There, the authors define two sets
\[
\mathscr{T}=\{(\xi,\eta): \varphi(\xi+\eta)=\varphi(\xi) + \varphi(\eta)\}, \quad 
\mathscr{S}=\{(\xi,\eta): \nabla \varphi(\xi) = \nabla \varphi(\eta)\}; 
\]
here $\varphi$ is the function describing the characteristic hypersurface, which, in our particular case, is $\varphi(\xi)=|\xi|^2$. The set $\mathscr{T}$ is called the time resonance set, and  $\mathscr{S}$ is called the space resonance set. Based on these, one defines the space-time resonance set as their intersection, that is
\[
\mathscr{R} = \mathscr{T} \cap \mathscr{S}.
\]
The space time resonance method developed in \cite{GeMaSh-1} works very well assuming $\mathscr{R} = \emptyset$. This implies, for example, that $|N(\xi) \wedge N(\eta)| \ne 0$ for any $(\xi,\eta) \in \mathscr{T}$,  which is a weaker condition to satisfy than \eqref{ndeg}. However, the theory developed in this paper based on \eqref{ndeg} is able to handle more general data (in the scaling space) in the translation invariant setup. 

Next, we consider a high-high to low interaction in which $|\xi+ \eta| \les |\xi| \approx |\eta|$. If we normalize $|\xi| \approx 1$, then 
\[
|N(\xi) \wedge N(\eta) \wedge N(\xi+\eta)| \approx 1.   
\]
This will not play a role in our analysis because there is another phenomena that we highlight in this case: the incompatibility of $|\xi+\eta| \ll 1$ and $|\xi \cdot \eta| \leq \frac12 |\xi||\eta|$. Another way to phrase this is to notice that if $|\xi+ \eta| \ll 1 \approx |\xi| \approx |\eta|$, then $||\xi+\eta|^2 - \xi^2 - \eta^2| \approx 1$, and this corresponds to the fact that, in order for the $I$ to be non-zero, at least one factor needs to have modulation $\ges 1$. We can quantify this precisely as follows. Assume that $u$ and $v$ are localized at frequency $\approx 1$ and $w$ is localized at frequency $\ll 1$. Then we have
\begin{equation} \label{lmi}
\int Q_{\prec 1} u \cdot Q_{\prec 1} v \cdot \overline{Q_{\prec 1} w} dx dt =0.
\end{equation}

We note that this last phenomenon is the major difference between the $u^2$ nonlinearity and the $|u|^2=u \bar u$ nonlinearity. In the latter one the high-high to low interactions do not enjoy a modulation bound from below and this makes the analysis significantly more involved.

Next we move beyond free solutions and consider the more general setup of functions whose Fourier transforms are not necessarily supported on the paraboloid. In this case, it is important to chase the effect of modulations. The key observation here is that if all three factors have low modulation, then the interaction involves non-parallel frequencies. To make this precise, consider the expression
\begin{equation} \label{Ileqk}
I  = \int Q_{\prec k} u Q_{\prec k} v \overline{Q_{\prec k}  w} dx dt, 
\end{equation}
where $u$ is localized at frequency $\approx 2^k$ and $v,w$ are localize at frequency $\approx 1$. In the support of $\widehat{Q_{\prec k} u}$ we have that $\tau_1=\xi_1^2+m_1, |\xi_1| \approx 2^{k}, |m_1| \ll 2^{k}$, while in the support of $\widehat{Q_{\prec k} v}$ we have $ \tau_2=\xi^2_2+m_2, |\xi_2| \approx 1, |m_2| \ll 2^{k}$; the support of $\widehat{Q_{\prec k} w}$ has a similar characterization to that of $\widehat{Q_{\prec k} v}$. To obtain non-trivial contribution to $I$, we need to impose 
\[
|\tau_1+ \tau_2 - (\xi_1+\xi_2)^2 | \ll 2^k,
\]
which translates into $|-2 \xi_1 \cdot \xi_2 + m_1 +m_2| \ll 2^k$, hence $|\xi_1 \cdot \xi_2| \ll 2^k$ or $|\angle(\xi_1,\xi_2)-\frac{\pi}2| \ll 1$. Thus the interacting frequencies are orthogonal; this is meant to be read as the interacting frequencies of $Q_{\prec k} u$ and $Q_{\prec k} v$. As a consequence of this, within the space-time Fourier support of the terms involved in \eqref{Ileqk} a similar property to \eqref{ndeg} holds true; to be more precise, if $(\xi_1,\tau_1)$ is in the support of $\widehat{Q_{\prec k} u}$, $(\xi_2, \tau_2)$ is in the support of $\widehat{Q_{\prec k} v}$, then
\begin{equation} \label{ndegg}
|N_{m_1}(\xi_1,\tau_1) \wedge N_{m_2}(\xi_2,\tau_2) \wedge N_{m_3}(\xi_1+\xi_2, \tau_1 + \tau_2)| \approx  |\xi_1|,  
\end{equation}
where $N(\xi_1,\tau_1)$ is the normal to the translated paraboloid $\tau=\xi^2+m_1, m_1 = \tau_1-\xi_1^2$
and so on. The conclusion of the above is that, the result in \eqref{ndeg} generalizes from the free wave setup to more general functions who display appropriate modulation constraints. 

There is one last technical observation. Consider the following version of \eqref{Ileqk}:
\begin{equation} \label{Ileq2k}
I  = \int Q_{\prec 2k} u Q_{\prec 2k} v \overline{Q_{\prec 2k}  w} dx dt.
\end{equation}
The same computation above shows that if $(\xi_1,\tau_1)$ is in the support of $\widehat{Q_{\prec 2k} u}$, $(\xi_2, \tau_2)$ is in the support of $\widehat{Q_{\prec 2k} v}$, then the angular constraint above improves to $|\xi_1 \cdot \xi_2| \ll 2^{2k}$ or $|\angle(\xi_1,\xi_2)-\frac{\pi}2| \ll 2^k$.

We finish this section with the analysis of the nonlinearity $(\bar u)^2$. Here the relevant integral is 
\[
I  = \int \bar u \bar v \bar w dx dt. 
\]
We make the following claim: if $u,v,w$ are localized at frequency $\xi \approx 2^{k_1}, 2^{k_2}, 2^{k_3}$ respectively, then, with $k_{max}=\max\{k_1,k_2,k_3\}$,  
\begin{equation} \label{barmod}
\int \overline{Q_{\prec 2k_{max}} u} \cdot \overline{Q_{\prec 2k_{max}} v} \cdot \overline{Q_{\prec 2k_{max}} w} dx dt =0.
\end{equation}
Indeed assume that $(\xi,\xi^2+m_1), (\eta,\eta^2+m_2)$ are in the Fourier support of $Q_{\prec 2k_{max}} u$, respectively $Q_{\prec 2k_{max}} v$; their output modulation is then
\[
- \xi^2 - m_1 -\eta^2 - m_2 - (\xi+\eta)^2  =   -m_1 -m_2 - \xi^2 -\eta^2 - (\xi+\eta)^2, 
\]
which, in absolute value, is $\approx 2^{2k_{max}}$. A consequence of \eqref{barmod} is that the nonlinearity $(\bar u)^2$ does not produce resonances, and in fact that in the trilinear expression for $I$ above, it suffices to consider the case when at least one term has rather high modulation. 

\section{The main trilinear estimate}

Below we use the notation $k_{max}= \max(k_1,k_2,k_3)$ and let $k_{min}= \min(k_1,k_2,k_3)$. The use of notation $u_{k}$ (or $v_k$ etc.) is meant to indicate that $u_{k}$ is localized at frequency $2^{k}$ in the sense $u_{k}=\tilde P_{k} u_{k}$. Our main result in this section is the following.

\begin{proposition} \label{Ptri} 
For any $\epsilon>0$, the following hold true:
\begin{enumerate}[label=(\roman*)]
    \item If $k_3 > k_{min}$ or $|k_{min}-k_{max}| \leq 10 $,
    \begin{equation} \label{mtri}
        | \int u_{k_1} v_{k_2} \bar w_{k_3} dx dt| \leq C(\epsilon) 2^{\epsilon(k_{max}-k_{min})}
        2^{-\frac{k_{max}}2} \|u_{k_1}\|_{V^2_\Delta} \|v_{k_2}\|_{V^2_\Delta}  \|w_{k_3}\|_{V^2_\Delta};
    \end{equation}

    \item If $k_3=k_{min} \leq k_{max} -10$,
    \begin{equation} \label{mtri2}
        | \int u_{k_1} v_{k_2} \bar w_{k_3} dx dt| \leq C(\epsilon)  2^{(1-\epsilon)(k_{min}-k_{max})} 2^{-\frac{k_{max}}2} \|u_{k_1}\|_{V^2_\Delta} \|v_{k_2}\|_{V^2_\Delta}  \|w_{k_3}\|_{V^2_\Delta}.
    \end{equation}

    \item For all $k_1,k_2,k_3$,
    \begin{equation} \label{mtri3}
        | \int \bar u_{k_1} \bar v_{k_2} \bar w_{k_3} dx dt| \leq C(\epsilon)  2^{(1-\epsilon)(k_{min}-k_{max})} 2^{-\frac{k_{max}}2} \|u_{k_1}\|_{V^2_\Delta} \|v_{k_2}\|_{V^2_\Delta}  \|w_{k_3}\|_{V^2_\Delta}.
    \end{equation}

\end{enumerate}
\end{proposition}

The first two estimates are needed in the analysis of the $u^2$ component of the nonlinearity; while \eqref{mtri2} is more technical, \eqref{mtri} is far more involved as it contains the resonances and requires the use of the new trilinear restriction theory. The third estimate is needed in the analysis of the $\bar{u}^2$ component of the nonlinearity; while it is the strongest of the three estimates, its proof is rather short due to the lack of resonances.

We believe that, in most cases, the $\epsilon$-loss in the above estimates can be removed, but this would require substantially more work both in this paper and in our companion paper \cite{BeCh}. On the other hand, removing the $\epsilon$-loss above would have no effect on the particular application to \eqref{QNLS} and the statement of Theorem \ref{thm gwp}.

Assuming Proposition \ref{Ptri}, we can now proceed with the proof of Theorem \ref{thm gwp}. 

\begin{proof}[Proof of Theorem \ref{thm gwp}] To keep the exposition compact, assume $N(z)=z^2$; it is easy to adapt the argument below for the general case $N(z)=c_1 z^2 + c_2 \bar z^2$.

For fixed $z_0 \in \dot H^{-\frac12}$, we define the operator $T: \dot X^{-\frac12} \rightarrow \dot X^{-\frac12}$ as the solution of the equation
\[
(i \partial_t - \Delta) Tz= z^2, \quad z(0)=z_0.
\]
Below we establish  that $T: \dot X^{-\frac12} \rightarrow \dot X^{-\frac12}$ is a contraction. 
The first step is a qualitative observation. A direct estimate shows that
\[
\| P_{k_3} ( P_{k_1} z \cdot P_{k_2} z) \|_{L^\infty_t L^2_x} \les \|P_{k_1} z\|_{L^\infty_t L^\infty_x} 
\| P_{k_2} z\|_{L^\infty_t L^2_x} \les 2^{\frac{nk_1}2} \|P_{k_1} z\|_{V^2_\Delta} \|P_{k_2} z\|_{V^2_\Delta}.
\]
This guarantees that $P_{k_3} ( P_{k_1} z \cdot P_{k_2} z) \in L^1_{t,loc} L^2_x$, hence  $P_{k_3} T ( P_{k_1} z, P_{k_2} z) \in C_t L^2_x$, in particular it is a ruled function; this also provides the qualitative estimate $P_{k_3} T ( P_{k_1} z, P_{k_2} z) \in C_t \dot H^{-\frac12}$, which can then be recovered globally using the uniform estimates described below. 

From \eqref{mtri} and \eqref{mtri2} we obtain the following bound
\[
2^{-\frac{k_3}2} \| P_{k_3} T ( P_{k_1} z, P_{k_2} z) \|_{V^2_\Delta} \les 2^{(\frac12 - \epsilon)(k_{min}-k_{max})} 2^{-\frac{k_1}2}  \|P_{k_1} z \|_{V^2_\Delta} 2^{-\frac{k_2}2} \|P_{k_2} z \|_{V^2_\Delta}. 
\]
A direct summation argument shows that $T$ is bounded on $\dot X^{-\frac12}$. Moreover, we also obtain the contraction type bound
\[
\| Tz_1 - Tz_2 \|_{\dot X^{-\frac12}} \les (\| z_1 \|_{\dot X^{-\frac12}} + \| z_2 \|_{\dot X^{-\frac12}}) \| z_1-z_2 \|_{\dot X^{-\frac12}}. 
\]
Therefore, under the assumption that the initial data $\|z_0\|_{\dot H^{-\frac12}}$ is small enough, $T$ becomes a contraction in $\dot X^s$, which in turns leads to the existence of a fixed point $z$ which corresponds to a global solution of the equation \eqref{QNLS}. Scattering is a consequence of the fact that functions in $V^2_\Delta$ have well-defined limits as $t \rightarrow \pm \infty$. 

The rest of the claims in Theorem \ref{thm gwp} are standard and follow from similar arguments. 

\end{proof}

The rest of this section is dedicated to proving Proposition \ref{Ptri}. The expressions are symmetric in $k_1,k_2$ so we can freely assume $k_1 \leq k_2$. Thus in what follows we focus on the case $|k_2-k_3| \leq 10$ and one of them is $k_{max}$. In order to keep the numerology compact, it is convenient to rescale the estimate so that we have $k_{max}=0$. 

\subsection{Proof of \eqref{mtri2}}
Since $k_3=k_{min} \leq k_{max}-10$, it follows that $|k_1-k_2| \leq 4$. We rescale the estimate so that $k_2=0$, thus $|k_1| \leq 4$ and $k_3 \leq -10$. 
Then \eqref{lmi} implies that at least one term has to have modulation $\ges 1$. Assume that a high frequency term, say $u_{k_1}$ has high modulation. In that case we use \eqref{UVX} and \eqref{V2bil} to obtain
\[
\| Q_{\succeq 1} u_{k_1} \cdot v_{k_2} \cdot \bar{w}_{k_3} \|_{L^1_{t,x}} 
\les \| Q_{\succeq 1} u_{k_1} \|_{L^2_{t,x}} \| v_{k_2} \cdot \bar{w}_{k_3} \|_{L^2_{t,x}} \les \la k_3 \ra^2 2^{k_3} \|u_{k_1}\|_{V^2_\Delta} \|v_{k_2}\|_{V^2_\Delta}
\|w_{k_3}\|_{V^2_\Delta}. 
\]
A similar estimate is obtained if $v_{k_2}$ has the high modulation. Thus it remains to consider the term $\| \tilde P_{k_3}(Q_{\prec 1} u_{k_1} \cdot Q_{\prec 1} v_{k_2}) \cdot \overline{Q_{\succeq 1} w_{k_3}} \|_{L^1_{t,x}} $. We start with the bilinear estimate for free solutions
\begin{equation} \label{impr}
\|\tilde P_{k_3}(e^{-it\Delta} f \cdot e^{-it\Delta} g) \|_{L^2_{t,x}} \les 2^{k_3} \|f\|_{L^2_x} \|g\|_{L^2_x},
\end{equation}
where $f,g$ are supported at frequency $2^{k_1}$, respectively $2^{k_2}$ (recall that $2^{k_1} \approx 2^{k_2} \approx 1$). We postpone the proof of \eqref{impr} and continue with our argument. From \eqref{impr} we obtain
\[
\|\tilde P_{k_3}(z_{k_1} \cdot z_{k_2}) \|_{L^2_{t,x}} \les  2^{k_3} \|z_{1}\|_{U^2_\Delta} \|z_{2}\|_{U^2_\Delta}; 
\]
indeed it is easy to see that \eqref{impr} generalizes to atoms in $U^2_\Delta$. On the other hand, using the Bernstein inequality, we have
\[
\begin{split}
\|P_{k_3}(z_{k_1} \cdot  z_{k_2}) \|_{L^2_{t,x}} & \les 2^{\frac{k_3}2} 
\|P_{k_3}(z_{k_1} \cdot  z_{k_2}) \|_{L^2_{t}L^\frac32_x} 
\les 2^{\frac{k_3}2} \|z_{k_1}\|_{L^4_t L^3_x} \|z_{k_2}\|_{L^4_t L^3_x} \\
& \les 2^{\frac{k_3}2} \|z_{k_1}\|_{U^4_\Delta}  \|z_{k_2}\|_{U^4_\Delta}. 
\end{split}
\]
By interpolating between the two estimates above, and using the inclusion $U^a_\Delta \subset V^2_\Delta$ for any $a >2$,  we obtain
\[
\|P_{k_3}( z_{k_1} \cdot z_{k_2}) \|_{L^2_{t,x}} \les C_\epsilon 2^{k_3(\frac12+\frac{1-\epsilon}2)} \|z_{k_1}\|_{V^2_\Delta}  2^{k_2} \|z_{k_2}\|_{V^2_\Delta},
\]
This allows us to estimate our original term as follows
\[
\begin{split}
& \| \tilde P_{k_3}(Q_{\prec 1} u_{k_1} \cdot Q_{\prec 1} v_{k_2}) \cdot \overline{Q_{\succeq 1} w_{k_3}} \|_{L^1_{t,x}} \les \| \tilde P_{k_3}(Q_{\prec 1} u_{k_1} \cdot Q_{\prec 1} v_{k_2})  \|_{L^2_{t,x}} \| Q_{\succeq 1} w_{k_3} \|_{L^2_{t,x}} \\
 \les & C_\epsilon 2^{k_3(\frac12+\frac{1-\epsilon}2)} \|z_{k_1}\|_{V^2_\Delta}  \|z_{k_2}\|_{V^2_\Delta} \|w_{k_3}\|_{V^2_\Delta}. 
\end{split}
\]
For proving \eqref{impr}, we notice that it suffices to consider the case when $\hat f$ and $\hat g$ are localized in cubes in $\R^3$ of size $2^{k_3}$; the two cubes have to be symmetric with respect to the origin so that the support of $\hat f \ast \hat g$ is localized at frequency $2^{k_3}$.  The global estimate is then recovered using standard orthogonality arguments.

With the above localization in place, we use Plancherel so that we estimate $\hat f(\xi) \delta_{\tau=\xi^2} \ast \hat g(\eta) \delta_{\tau=\eta^2}$ in $L^2_{\tau,\xi}$. It is easy to see that the intersection of any translate of the support of $\hat f(\xi) \delta_{\tau=\xi^2}$ with the support of $g(\eta) \delta_{\tau=\eta^2}$ has $\mathcal{H}^{1}$ measure $\les 2^{2k_3}$ and the angle between any normal to the paraboloid in the support of $\hat f(\xi) \delta_{\tau=\xi^2}$ and any normal to the paraboloid in the support of $\hat g(\xi) \delta_{\tau=\xi^2}$ is bounded below by $\ges 1$. As a consequence of \eqref{bilGe} we have
\[
\| \hat f(\xi) \delta_{\tau=\xi^2} \ast \hat g(\eta) \delta_{\tau=\eta^2} \|_{L^2} \les 2^{k_3} \| f\|_{L^2} \|g\|_{L^2}. 
\]
This concludes the argument for \eqref{impr}. 

\subsection{Proof of \eqref{mtri3}} We recall \eqref{barmod}, thus it suffices to consider the case when at least one term has modulation $\succeq 1$. But then the argument is identical to the one used in the proof of \eqref{mtri2}. 

\subsection{Proof of \eqref{mtri}} The estimate is symmetric in $k_1$ and $k_2$, hence we can freely assume that $k_1 \leq k_2$. The additional constraints imply that it suffices to consider the case $|k_2-k_3| \leq 10$. By rescaling it suffices to consider the case $k_3=0$; to keep the number of parameters at a minimum we can assume that $k_2=0$ and $k_1 \leq 5$; the particular choice $k_2=0$ plays no role and the argument stays the same if $|k_2| \leq 10$. We also relabel $k_1=k$. 

Thus for the rest of this section the setup will be as follows. The functions $v,w$ are localized at frequency $\approx 1$, while the function $u$ is localized at frequency $\approx 2^k \les 1$. We skip the use of $u_k, v_0,w_0$ to indicate the frequency localization, and simply use $u,v,w$.  

The argument is split into two main parts, based on whether at least one term has high modulation or all of them have low modulation. Our first claim is about the case when at least one term has high modulation. 
\begin{lemma} \label{l:trihm}
The following holds true:
  \begin{equation} \label{trihmh}
    \| u \cdot Q_{\succeq 2k}v \cdot \bar w \|_{L^1} + \| u \cdot v \cdot
    \overline{ Q_{\succeq 2k} w} \|_{L^1} \les 
    \langle k\rangle^2
    \|u\|_{V_{ \Delta }^{ 2 }} \|v\|_{V_{ \Delta }^{ 2 }} \|w\|_{ V_{ \Delta
    }^{ 2 }}, 
  \end{equation}
  \begin{equation} \label{trihml}
    \| Q_{\succeq 2k} u \cdot v \cdot \bar w \|_{L^1} 
    \les \langle k \rangle^3  
    \|u\|_{ V_{ \Delta }^{ 2 }} 
    \|v\|_{ V_{ \Delta }^{ 2 }} 
    \|w\|_{ V_{ \Delta }^{ 2 }}.
  \end{equation}
\end{lemma}

Our second claim is about the case when all terms have low modulation.
\begin{theorem}\label{trilml} For any $\varepsilon>0$, there exists $C_\varepsilon>0$ such that the following holds true
\begin{equation}\label{eq transferred trilinear estimate}
| \int Q_{\prec 2k} u \cdot Q_{\prec 2k} v \cdot \overline{ Q_{\prec 2k} w} \ dx dt| 
\leq C_\varepsilon 2^{-k\varepsilon} \|u\|_{V_{ \Delta }^{ 2 }} \|v\|_{V_{ \Delta }^{ 2 }} \|w\|_{V_{ \Delta }^{ 2 }}.
\end{equation}
\end{theorem}

Lemma \ref{l:trihm} is more ``classical'' in that the existing techniques can be deployed given that one factor has high modulation. This does not mean that the argument is trivial, as we discovered that \eqref{trihml} would come with a worse factor if we had relied solely on classical methods. This is the reason we have the additional Lemma \ref{prop:ref-strichartz-2} below. 

On the other hand Theorem \ref{trilml} is completely new in the literature and it provides the effective transference of the trilinear estimate for free solutions to our resolution spaces. Note that Theorem \ref{trilml} implies Theorem \ref{thm trilinear transference} as a special case when $2^k \approx 1$. The use of $\prec$ here versus $\les$ in  Theorem \ref{thm trilinear transference} is not important; we do so in order  to extract the correct transversality assumptions \eqref{Nres}, which are already assumed to be true in Theorem \ref{thm trilinear transference}.

\begin{proof}[Proof of \eqref{trihmh}]
   We use \eqref{V2bil} and \eqref{UVX} to obtain:
  \begin{align*}
      \lVert u v \overline{Q_{ \succeq 2k } w} \rVert _{ L^{ 1 } } 
    &\lesssim \lVert u v \rVert _{ L^{ 2 } } \lVert Q_{ \succeq 2k
    }w \rVert _{ L^{ 2 } } 
    \lesssim 2^{ k } \langle k \rangle^2 \lVert u \rVert _{ V_{ \Delta }^{ 2 } } 
    \lVert v \rVert _{ V_{ \Delta }^{ 2 } } 
    \cdot 2^{ -k } \lVert w \rVert_{ \dot X^{ 0,\frac{1}{2},\infty } }\\
    &\lesssim \langle k \rangle^2
    \lVert u \rVert _{ V_{ \Delta }^{ 2 } }
    \lVert v \rVert _{ V_{ \Delta }^{ 2 } }
    \lVert w \rVert _{ V_{ \Delta }^{ 2 } }.
  \end{align*}

\end{proof}

\begin{proof}[Proof of \eqref{trihml}] Here the argument is a more complex.
There is an easy part, when the low frequency has a higher modulation localization. Indeed, using Bernstein, \eqref{V2Str} and \eqref{UVX}, we estimate as follows:
  \begin{equation} \label{easy}
  \begin{split}
    \left\lvert \int_{ \R^{ 1+3 } } Q_{ \succeq k }u v \bar w
      dx dt \right\rvert 
    &\lesssim \lVert  Q_{ \succeq k } u  \rVert_{ L^{ 2 }_{ t,x } } 
    \lVert \widetilde P_{ k } (v \bar w) \rVert_{ L_{ t,x }^{ 2 }  }
    \lesssim 2^{ -\frac{k}{2}  } \lVert u \rVert _{ \dot X^{ 0,\frac{1}{2}
      ,\infty } } 
      \cdot 2^{ \frac{k}{2}  } \lVert v \bar w \rVert _{ L_{ t }^{ 2 } L_{ x }^{
      \frac{3}{2} }  }\\
    &\lesssim \lVert u \rVert _{ V_{ \Delta }^{ 2 } } 
      \lVert v \rVert _{ L_{ t }^{ 4 } L_{ x }^{ 3 }  } 
      \lVert w \rVert _{ L_{ t }^{ 4 } L_{ x }^{ 3 }  } 
      \lesssim \lVert u \rVert_{ V_{ \Delta }^{ 2 } }
      \lVert v \rVert_{ V_{ \Delta }^{ 2 } }
      \lVert w \rVert_{ V_{ \Delta }^{ 2 } }.
      \end{split}
  \end{equation}
  For modulations $2k \preceq j \preceq k$ we change the strategy. Fix such a $j$. We notice that
  \[
  \int_{ \R^{ 1+3 } } Q_{ j }u \cdot v \bar w  d xd t 
  = \sum_{q,q' \in \mathcal{C}_k} \int Q_{ j }u \cdot v_q \bar w_{q'}  d xd t,
  \]
  where the sum is restricted to the pairs with $d(q,q') \les 2^{k}$; here we use the notation $v_q = \chi_q(D) v$ for the component of $v$ with the property that $\hat v$ is localized in $q$, and similarly for $w_{q'}$. Then we fix such a pair $(q,q')$. We note that also have
  \[
  \int Q_{ j }u \cdot v_q \bar w_{q'}  d xd t = \sum_{I,I' \in \mathcal{I}_{j}} \int Q_{ j }u \cdot v_{q,I} \bar w_{q',I'},
  \]
where $ v_{q,I} = \chi_I(D_t) v_{q}$ is the component of $v_q$ with the property that $\hat v_q(\xi,\tau)$ has $\tau$ localization in $I$. Since $Q_j u$ is localized in the region $|\tau| \les 2^j$, 
it suffices to restrict the summation over  pairs of intervals with the property $d(I,I') \les 2^{j}$. We remark that, for a free solution
$(e^{-it\Delta} f)_q=e^{-it\Delta} f_q$, while the time frequency localization translates into $(e^{-it\Delta} f_q)_I= e^{-it\Delta} \chi_I(\xi^2) f_{q}$; essentially this produces a localization of $\xi$ in $S$, the strip characterized by $\xi^2 \in I$. This is a subset of the parallelepiped of size $C2^j \times C2^k \times C2^k$, where the size of $C2^j$ comes in the direction of $\xi_0$, where $\xi_0$ is the center of $q$.

\begin{lemma}\label{prop:ref-strichartz-2} 
i) Assume that $j \leq k \leq 0$ and $S$ is parallelepiped centered at $\xi_0$, $|\xi_0| \approx 1$ and of dimensions $ \approx 2^{ j }\times 2^{ k } \times 2^{ k }$. If $\hat f$ is supported in $S$, then
  \begin{align}\label{eq:slab-strichartz-2}
    \|e^{-it\Delta}f\|_{L_t^pL_x^q}
    \lesssim
    2^{\frac{2k+j}2 \left(1-\frac{2}{q}\right) - \frac{2k}{p}}
    \|f\|_{L_x^2}
  \end{align}
  whenever $2<p\leq\infty, \quad 2\leq q\leq\infty,\quad    \frac1p+\frac1q\leq\frac12$.
  In particular, $\|e^{-it\Delta}f\|_{L_{t,x}^4} \lesssim 2^{\frac{j}4}\|f\|_{L_x^2}$.

ii) Assume $u \in U^4_\Delta$ has similar localization properties. Then
\begin{equation}
\| u \|_{L^4_{t,x}} \les 2^\frac{j}4 \|u\|_{U^4_\Delta}. 
\end{equation}

\end{lemma}

Assuming the above lemma, we conclude  the proof of \eqref{trihml}. For fixed pair $(q,q')$ and $I,I'$ we have
\begin{align*}
    \left\lvert \int_{ \R^{ 1+3 } } Q_{ j }u \cdot v_{q,I} \bar w_{q',I'}  d xd t \right\rvert
    &\lesssim \lVert Q_{ j }u  \rVert_{ L_{t,x}^{2 }}
      \lVert v_{q,I} \rVert_{ L_{ t,x }^{ 4 }  } \lVert w_{q',I'} \rVert_{ L_{ t,x } ^{ 4 }(T) }\\
    & \lesssim 2^{-\frac{j}2} \|u\|_{V^2_\Delta} 2^\frac{j}4 \|v_{q,I}\|_{U^4_\Delta} 2^\frac{j}4 \|w_{q',I'}\|_{U^4_\Delta} \\
& \lesssim  \|u\|_{V^2_\Delta} \|v_{q,I}\|_{U^4_\Delta} \|w_{q',I'}\|_{U^4_\Delta}. 
  \end{align*} 
Then we estimate the global term using Lemma \eqref{l2V2} as follows
\begin{align*}
    \left\lvert \int_{ \R^{ 1+3 } } Q_{ j }u \cdot v \bar w  d xd t \right\rvert
    & \lesssim \|u\|_{V^2_\Delta} \left(\sum_q \sum_I \|v_{q,I}\|^2_{U^4_\Delta} \right)^\frac12 \left( \sum_{q'} \sum_{I'} \|w_{q',I'}\|^2_{U^4_\Delta} \right)^\frac12 \\
    & \les \|u\|_{V^2_\Delta} \la j \ra^2 \|v\|_{V^2_\Delta} \|w\|_{V^2_\Delta}.
  \end{align*} 
Strictly speaking, the result in Lemma \eqref{l2V2}, does not apply verbatim above, and requires some slight  adjustments. First, the localization of $\chi_q(D_x) \chi_I(D_t)$ is not entirely in spatial variables, as required by 
Lemma \eqref{l2V2}. However, recalling that we work with $Q_{\preceq 2k} v$ and $Q_{\preceq 2k} w$, then, just as we did for free solutions, this translates into a parallelepiped localization solely on the spatial variables. Second, the statement Lemma \eqref{l2V2} is made in terms of isotropic objects (cubes to be more precise), while here the localization is made on parallelepipeds. This is again not a problem, since the argument Lemma \eqref{l2V2} provides the argument in this setup as well, and we made sure we picked up the correct factor $\la j \ra$ since $2^j$ is the smallest scale that appears in the parallelepiped.

The summation over the range $2k \preceq j \preceq k$ of the above estimate, combined with \eqref{easy}, gives \eqref{trihml}. 

Thus we are left with establishing the claim in Lemma \ref{prop:ref-strichartz-2}, which we prove below.

\end{proof}

\begin{proof}[Proof of Lemma \ref{prop:ref-strichartz-2}]
  The case $p=\infty$ follows from the anisotropic Bernstein bound
  \[
    \|e^{-it\Delta}f\|_{L_t^\infty L_x^q}
    \lesssim
    |S|^{\frac12-\frac1q}\|f\|_{L_x^2}
    \lesssim
    2^{2k(1-\frac2q)}\|f\|_{L_x^2}.
  \]
  We may therefore assume that $p<\infty$. For simplicity, we assume that $S$
  is centered at $e_1$. We let $Af= e^{-it\Delta}\chi (D) f$,
  where $\chi$ is a smooth cutoff to a slightly enlarged copy of $S$.
  The operator $AA^*$ is a spacetime convolution operator with kernel
  \[
    K_k(t,x)
    =
    \int_{\mathbb R^3}
      e^{i(x\cdot\xi+t|\xi|^2)}
      \chi(\xi)^2\,d\xi.
  \]
  By the $TT^*$ argument, it suffices to prove
  \[
    \|AA^*\|_{L_t^{p'}L_x^{q'}\to L_t^pL_x^q}
    \lesssim 2^{(2k+j) \left(1-\frac{2}{q}\right) - \frac{4k}{p}}.
  \]
  For this we use the pointwise kernel bound
  \begin{equation}\label{eq:anisotropic-kernel}
    |K_k(t,x)|
    \lesssim
    2^{2k+j}(1+2^{2k}|t|)^{-1}.
  \end{equation}
  Interpolating the resulting $L_x^1\to L_x^\infty$ estimate with the
  $L_x^2\to L_x^2$ bound and then using Young's inequality in $x$, we
  obtain
  \[
    \|K_k(t,\cdot)*_x\phi\|_{L_x^q}
    \lesssim    2^{(2k+j)(1-\frac2q)}
    (1+2^{2k}|t|)^{-(1-\frac2q)}
    \|\phi\|_{L_x^{q'}}.
  \]
  If $\frac1p+\frac1q=\frac12$, then
  $1-\frac2q=\frac2p$, and the Hardy--Littlewood--Sobolev inequality in
  the time variable gives
  \[
    \|AA^*F\|_{L_t^pL_x^q}
    \lesssim
    2^{(2k+j)\left(1-\frac{2}{q}\right)} (2^{-2k})^{\frac2p}
    \|F\|_{L_t^{p'}L_x^{q'}}.
  \]
  If $\frac1p+\frac1q<\frac12$, the same conclusion follows from
  Young's inequality in the time variable, since
  $1-\frac2q>\frac2p$. This proves the required $AA^*$ bound.

  It remains to prove \eqref{eq:anisotropic-kernel}. Write
  \[
    \xi_1=1+2^{2k}\rho,
    \qquad
    \xi'=2^k\eta,
    \qquad
    \xi'=(\xi_2,\xi_3).
  \]
  After removing a phase of absolute value one, we obtain
  \[
    K_k(t,x)
    =
    2^{4k}
    \int_{\mathbb R}
      e^{i(2^{2k}(x_1+2t)\rho+2^{4k}t\rho^2)}
      \int_{\mathbb R^2}
        e^{i(2^kx'\cdot\eta+2^{2k}t|\eta|^2)}
        a(\rho,\eta)\,d\eta\,d\rho,
  \]
  where $a$ is smooth, supported in a fixed bounded set, and has uniform
  symbol bounds. 
  By stationary phase in the variable $ \eta $,
  \[
    \left|
      \int_{\mathbb R^2}
        e^{i(y\cdot\eta+s|\eta|^2)}
        a(\rho,\eta)\,d\eta
    \right|
    \lesssim
    (1+|s|)^{-1},
  \]
  uniformly in $\rho$ and $y$. Taking $s=2^{2k}t$ and
  $y=2^kx'$, and estimating the remaining $\rho$-integral directly,
  yields \eqref{eq:anisotropic-kernel}.

  The transference to $ U_{ \Delta }^{ 4 } $ follows immediately for atoms: if
  $ a(t) = \sum_{ J }  \1_{ J }e^{ -it\Delta } \phi_{ J }   $ is a $ U_{ \Delta
  }^{ 4 } $-atom,
  \begin{align*}
    \lVert a \rVert _{ L_{ t,x }^{ 4 }  } ^{ 4 }
    &\lesssim \sum_{ J } \int_{ J }\lVert e^{ -it\Delta } \phi_{ J }  \rVert _{
    L_{ t,x }^{ 4 } }^{ 4 }
    \lesssim 2^{ j } \sum_{ J } \lVert \phi_{ J }  \rVert_{ L_{ t,x }^{ 4 }
    }^{ 4 },
  \end{align*} 
  since the time intervals $ J $ are disjoint.
  The estimate for generic $ U_{ \Delta }^{ 4 } $ functions then follows by
  density.

\end{proof}

\section{The trilinear restriction estimate and the proof of Theorem \ref{trilml}}

In this section, we prove the estimate in Theorem \ref{trilml} by transferring the trilinear restriction estimate in $\R^4$ to $V^2_\Delta$. We recall the estimate here for convenience:
\[
| \int Q_{\prec 2k} u \cdot Q_{\prec 2k} v \cdot \overline{ Q_{\prec 2k} w} \ dx dt| 
\leq C_\varepsilon 2^{-k\varepsilon}\|u\|_{V^2_\Delta} \|v\|_{V^2_\Delta} \|w\|_{V^2_\Delta},
\]
where  $u$ is localized to frequency $2^k$, $v$ and $w$ are localized to frequency $1$, and $k < 0$.

The transference argument for low modulation waves that we use is based on ideas introduced by the first author in \cite{Be-rs}. The idea is to interpolate between the transferred estimate for functions in $U^2_\Delta$, which follows from a fairly standard argument, and $U^a_\Delta$, which follows from a new linear estimate that was proved in \cite{Be-rs}. A key point is that the interpolation is done for certain $\ell^p L^q$ quasi-norms that reflect the geometry of the wave propagation, instead of standard $L^p$ norms.

Thus, in order for the transference to work we need to have the corresponding trilinear restriction estimate in $\ell^p L^q$ with the correct localization hypotheses. This is proved by the first and third authors in a companion paper \cite{BeCh} in a more general geometric set-up that allows for degenerate curvatures in the underlying surfaces. This extra flexibility is crucial for the arguments in this section when $2^k\ll 1$, which will be described in more detail later.

We organize the proof of Theorem \ref{trilml} into two subsections corresponding to these two main steps as the arguments involved in each step are quite different in nature from each other (and from the rest of this paper). First, we state the precise free wave estimate that we need and then transfer the low modulation components to $V^2_\Delta$. Second, we prove this free wave estimate by transforming the estimate into a set-up in which we can apply the more general result from \cite{BeCh} to conclude.

\subsection{Transference argument}

In order to state the free wave estimate that we wish to transfer, we first need some technical reductions. Given the localization of $u$ at frequency $2^{k}$, it suffices to further localize the terms $Q_{\prec 2k} v$ and $Q_{\prec 2k} w$ in cubes of size $\approx 2^{k}$. Precisely we have
\[
\int Q_{\prec 2k} u \cdot Q_{\prec 2k} v \cdot \overline{ Q_{\prec 2k} w} \ dx dt = 
\sum_{q,q' \in \mathcal{C}_{k-5}} \int Q_{\prec 2k} u \cdot Q_{\prec 2k} v_{q} \cdot \overline{ Q_{\prec 2k} w_{q'}} \ dx dt,
\]
and we can further restrict the sum to cubes $q,q' \in \mathcal{C}_{k-5}$ that satisfy the condition $d(q,q') \approx 2^{k}$ - this is to ensure that the term $Q_{\prec 2k} v_{q} \cdot \overline{ Q_{\prec 2k} w_{q'}}$ is localized at frequency $\approx 2^k$.

Next, we fix a pair of cubes $q,q'$ as above.  Without restricting the generality of the argument, we assume that the center of $q$ satisfies $c(q)=(0,0,1)= e_3 $. Thus, if $\eta \in q$, then $|\eta \cdot e_3 - |\eta|| \leq 2^{k-3}$.

Given the modulation localization, the computation for \eqref{Ileq2k} shows that we can further localize the low frequency term $Q_{\prec 2k} u$ in a region where $|\xi \cdot e_3| \ll 2^{2k}$ or, equivalently $|\angle (\xi,e_3)-\frac{\pi}2| \ll 2^k$. As a consequence we obtain $|(c(q')-c(q)) \wedge e_3| \approx 2^{k}$, which implies that the separation between $q$ and $q'$ (i.e. ($d(q,q') \approx 2^{k}$) comes in directions transversal to $e_3$.  
 
For the rest of this section, we implicitly assume in what follows that in the integral 
\[
I_{q,q'} = \int Q_{\prec 2k} u \cdot Q_{\prec 2k} v_{q} \cdot \overline{ Q_{\prec 2k} w_{q'}} \ dx dt,
\]
the terms are localized as above.

Based on this setup we define $\mathcal{P}$ to be the set of spacetime parallelepipeds in $\R^4$ with sides parallel to the coordinate axes $e_1,e_2,e_3$ and $e_4$ (the last being the time direction), and whose sides have length $2^{-2k}$ in the directions of $e_3$ and $e_4$, and $2^{-k}$ in the directions of $e_1$ and $e_2$. Explicitly, a parallelepiped $P \in \mathcal{P}$ is a spacetime set of the form
\[
P= P_z= \{ x \in \R^4:  x_i \in [2^{-k} z_i , 2^{-k}(z_i+1)], i=1,2;
 x_i \in [2^{-2k} z_i , 2^{-2k}(z_i+1)], i=3,4 \},
\]
where $z \in \Z^4$ and we write $x=(x_1,x_2,x_3,x_4)$. 

To justify this construction, we note that a parallelepiped $P \in \mathcal{P}$ has the property that it is the maximal (in size) geometric structure which accommodates wave packets associated to the free waves in the expression $I_{q,q'}$ before dispersion occurs. Indeed, when constructing wave packets for the Schr\"odinger equation, the localization in frequency at scale $\Delta \xi \approx 2^{k}$ (displayed by all terms in $I_{q,q'}$) calls for a localization at spatial scale $\Delta x \approx 2^{-k}$ and time scale $2^{-2k}$. The wave packets associated to $Q_{\prec 2k} u$ are essentially vertical as their velocity is close to $e_4=(0,0,0,1)$, and therefore stay coherent up to scale $2^{-2k}$ in the $x_4$ direction. On the other hand, the wave packets associated to $Q_{\prec 2k} v_{q}$ and $Q_{\prec 2k} w_{q'}$ essentially travel in the direction $(0,0,-2,1)$, and therefore stay coherent up to scale $2^{-2k}$ in the $x_3$ direction. This determines the geometry of the parallelepipeds in $\mathcal{P}$.

Using this notation, we introduce the norm or quasi-norm
\[
\|F\|_{\ell^p_{\mathcal{P}} L^1}: = \left( \sum_{P \in \mathcal{P}} \| F \|_{L^1(P)}^p\right)^\frac1p, 
\]
for $F: \R^4 \rightarrow \C$. The main goal in this section is then to transfer the following estimate for free solutions.
\begin{theorem} \label{ltri1}
Let $p>\frac{14}{15}$ and assume $f,g,h$ are frequency-localized as above. Then for any $\epsilon>0$ there exists $C_\epsilon>0$ such that
\begin{equation} \label{triro}
\|e^{-it\Delta} f \cdot e^{-it\Delta} g \cdot \overline{e^{-it\Delta} h} \|_{\ell^p_{\mathcal{P}} L^1} \leq C_\epsilon 2^{-k\epsilon} \|f\|_{L^2} \|g\|_{L^2} \|h\|_{L^2}.
\end{equation}
\end{theorem}
We state the estimate with the conjugate on the third term in order to match the statement of \eqref{eq transferred trilinear estimate}. However we note that, given the extra localizations in place, the conjugate has no effect on the estimate. Thus, for the rest of this section and the appendix, we will omit the conjugate from the estimates for clarity.

Before we put this result to work, let us unpack its statement. At the level of each parallelepiped $P$, we measure the interaction in $L^1$, which is the natural space for generic trilinear interactions - see \eqref{trint} and the comment right after. In particular, the curvature of the paraboloid does not provide any significant dispersion within the parallelepipeds and therefore we cannot do any better than the generic trilinear restriction estimate, i.e. an $L^1(P)$ estimate. Once we move outside the parallelepipeds, the improved $\ell^p$ summation reflects the fact that we are in $\R^4$ and have additional curvature which improves the estimate. We refer to this as an improved summability because if $p<1$ then this estimate is stronger than the generic global estimate in $l^1_{\mathcal{P}} L^1=L^1(\R^n)$.

We note that when $2^k\approx 1$, we could simply use Theorem \ref{TSS} instead, which is the more standard estimate in $L^p(\R^n)=\ell^p_{\mathcal{P}}L^p$ for the same range of $p$,
\begin{equation}
    \norm{e^{-it\Delta} f_1 \cdot e^{-it\Delta} f_2 \cdot e^{-it\Delta} f_3}_{L^p(\R^4)}\lesssim \norm{f_1}_{L^2} \norm{f_2}_{L^2} \norm{f_3}_{L^2}.
\end{equation}
However, when $2^k\ll 1$, we need the full power of Theorem \ref{ltri1} for the transference because of the degenerate geometry of this set-up. In fact, because the optimal estimate within the parallelepipeds is an $L^1$ estimate, as described above, using $L^p$ with $p<1$ would actually incur additional losses in $2^{-k}$ in this case.

Theorem \ref{ltri1} will be proved in the next subsection. Assuming this result for now, we proceed with the proof of Theorem \ref{trilml}. We use the transference argument from \cite{Be-rs} which interpolates estimates in $U^2_\Delta$ and $U^a_\Delta$ ($a>2$) in order to obtain a final estimate in $V^2_\Delta$. In particular, this argument uses the $\ell^p L^1$ structure in an important way, which is another reason why we need to write the trilinear estimate \eqref{triro} in this form.

\begin{proof}[Proof of Theorem \ref{trilml}]
First, a standard argument (following \cite[Theorem 7]{Ca-note}) shows that \eqref{triro} can be generalized to functions in $U^2_\Delta$, that is
\begin{equation} \label{me1}
\|u \cdot v_q \cdot \bar w_{q'} \|_{\ell^p_{\mathcal{P}} L^1} \leq C_\epsilon 2^{-k\epsilon} \|u\|_{U^2_\Delta} \|v_q\|_{U^2_\Delta} \|w_{q'}\|_{U^2_\Delta},
\end{equation}
where, again, we assume all the localization mentioned above to be in place. We provide a complete proof of this transference in the appendix.

The second step in the transference is the following estimate in $U^a_\Delta$
\begin{equation}  \label{me2}
\| Q_{\prec 2k} u \cdot Q_{\prec 2k} v_q \cdot \overline{Q_{\prec 2k} w_{q'}} \|_{l^\infty_{\mathcal{P}} L^1} \leq C_\epsilon 2^{-k\epsilon} \|u\|_{U^a_\Delta} \|v_q\|_{U^a_\Delta} \|w_{q'}\|_{U^a_\Delta}, 
\end{equation}
for any $2 \leq a < \infty$. To prove this we use the following linear estimate: if $I$ is a time interval of length $\approx 2^{-2k}$ and $\varphi_I$ is a bump function adapted to $I$, then
\begin{equation} \label{sme}
\| \varphi_I Q_{\prec 2k} z \|_{X^{0,\frac12,1}} \les \| z \|_{U^a_\Delta},
\end{equation}
The proof can be found in \cite[Lemma 5.5]{Be-rs} combined with the embedding $U^a_\Delta \subset L^\infty_t L^2_x$. This estimate allows us to treat each of the terms $Q_{\prec 2k} u$, $Q_{\prec 2k} v_q$, $Q_{\prec 2k} w_{q'}$ as elements of $X^{0,\frac12,1}$ when restricted to a single parallelepiped $P$. Since $X^{0,\frac12,1} \subset U^2_\Delta$, we then invoke \eqref{me1} to obtain \eqref{me2}.

Next, we interpolate between \eqref{me1} (with some $\frac{14}{15}<p<1$) and \eqref{me2} to obtain
\begin{equation} \label{mef}
\|Q_{\prec 2k} u \cdot Q_{\prec 2k} v_q \cdot \overline{Q_{\prec 2k} w_{q'}} \|_{l^1_{\mathcal{P}} L^1} \leq C_\epsilon 2^{-k\epsilon} \|u\|_{U^a_\Delta} \|v_q\|_{U^a_\Delta} \|w_{q'}\|_{U^a_\Delta},
\end{equation}
where $a=\frac{2}{p}>2$. The formalization of this interpolation can be found in the proof of Theorem 1.7 in \cite{Ca}

This gives the correct global bound in $L^1$. The next step is the summation with respect to the cubes $q,q'$. Here we have
\[
\begin{split}
|I| \leq C_\epsilon 2^{-k\epsilon} \sum_{q,q' \in \mathcal{C}_k} \|u\|_{U^a_\Delta} \|v_q\|_{U^a_\Delta} \|w_{q'}\|_{U^a_\Delta}
\les \|u\|_{U^a_\Delta} \left( \sum_{q \in \mathcal{C}_k}  \|v_q\|_{U^a_\Delta}^2 \right)^\frac12
\left( \sum_{q' \in \mathcal{C}_k}  \|w_{q'}\|_{U^a_\Delta}^2 \right)^\frac12,
\end{split}
\]
where this was possible because the sum $\sum_{q,q' \in \mathcal{C}_k}$ is restricted over cubes with the property $d(q,q') \les 2^k$. Invoking Lemma \ref{l2V2} and the embedding $V^2_\Delta\subset U^a_\Delta$ allows us to conclude
\[
|I| \leq C_\epsilon 2^{-k\epsilon} \la k \ra^2 \|u\|_{V^2_\Delta} \|v\|_{V^2_\Delta} \|w\|_{V^2_\Delta} \lesssim C_\epsilon 2^{-k\epsilon} \|u\|_{V^2_\Delta} \|v\|_{V^2_\Delta} \|w\|_{V^2_\Delta}
\]
which proves Theorem \ref{trilml} with an acceptable subpolynomial loss.
\end{proof}

\subsection{Proof of Theorem \ref{ltri1}}

Given the localizations in place, if $\xi, \eta, \zeta$ are in the support of $\hat f, \hat g, \hat h$ respectively, then by \eqref{ndeg} we have the following transversality condition
\begin{equation} \label{NN}
    |N(\xi) \wedge N(\eta) \wedge N(\zeta)| \approx  |\xi| \approx 2^k. 
\end{equation}
where we recall that $N(\xi)$ stands for the unit normal vector to the paraboloid at the point $(\xi,\xi^2)$. When $2^k\ll 1$, this represents a setup with degenerate transversality in the trilinear restriction estimate. Heuristically, based on similar considerations as \eqref{bilGe} for the bilinear estimate, this should then lead to an additional loss of $2^{-\frac{k}{2}}$ in the trilinear estimate. However, as noted in the discussion following \eqref{Ileq2k}, we have an additional localization at scale $2^k$; the key point is that we can leverage this to gain back a factor of $2^{\frac{k}{2}}$ and perfectly cancel out the loss from the transversality constant, which then gives us \eqref{triro}.

To formalize this argument we use a standard parabolic rescaling argument (e.g. as seen in \cite[Proposition 2.6]{tao_bilinear_1998}) in order to transform the problem into a set-up with transversality constant $\approx 1$. Consequently, the low frequency support will be sent to a cap on the paraboloid far from the origin with degenerate principal curvatures, at which point we will be able to apply a result from \cite{BeCh} to conclude the proof. We found this to be the easiest approach to proving \eqref{triro} but it is likely that other approaches could work too.

First, we run the parabolic rescaling argument and track the effects on the supports, parallelepipeds, and the numerology of the estimate. We recall the localizations of the factors involved:
\begin{itemize}
    \item $\hat g$ is localized in the cube $q$, $\hat h$ is localized in the cube $q'$, and $q,q' \in \mathcal{C}_{k-5}$ with $d(q,q') \approx 2^k$; we normalized $c(q)=e_3$ and argued that $|(c(q')-c(q)) \wedge e_3| \approx 2^{k}$.

    \item $\hat f$ is localized within the set $\{\xi \in \R^3: |\xi| \in [2^{k-1},2^{k+1}], |\angle (\xi,e_3)-\frac{\pi}2| \ll 2^k \}$
\end{itemize}
In this set-up we have the transversality condition \eqref{NN}.

Given $\gamma \in\R^3$ and $\delta>0$, we define the translation and scaling operators:
\begin{align*}
  T_\gamma f (\xi)= f(\xi-\gamma), \quad S_\delta f(\xi) = f(\delta^{-1}\xi).     
\end{align*}
We will slightly abuse notation and apply these operators to sets as well. 

Let $\gamma=e_3$ and let $L=S_{2^k}T_\gamma$. Then $S_{2^k} T_\gamma q$ is a cube of size $\approx 1$ centered at the origin and $S_{2^k} T_\gamma q'$ is a cube of size $\approx 1$ at scale $\approx 1$. On the other hand, $S_{2^k} T_\gamma$ moves the support of $\hat f$ to a subset of $\{\xi \in \R^3: |\xi_1| + |\xi_2| \les 1, \xi_3 \in [ l, l+2^k C]\}$ for some $\abs{l} \approx 2^{-k}$ and some $C >0$; in particular it is at scale $\approx 2^{-k}$. A direct computation then shows that if $\xi, \eta, \zeta$ are in the support of 
$L(\supp\hat f), L(\supp\hat g), L(\supp\hat h)$ respectively, then
\begin{equation} \label{Ngood}
|N(\xi) \wedge N(\eta) \wedge N(\zeta)| \approx 1.
\end{equation}

Next, we quantify how the estimate \eqref{triro} changes under this transformation. For this we fix some notation. As the rescalings primarily affect the Fourier supports of the functions, we use the extension operator notation for the rest of this section
\begin{align}\label{eq extension operator}
\mathcal{E} \hat f(x,t)\defeq e^{-it\Delta} f(x) = \int e^{i(x\cdot\xi+t \xi^2)} \hat f(\xi) d\xi.
\end{align}
For $\psi \in \{\hat f, \hat g,\hat h \}$, we let $\psi_L = \psi \circ L^{-1}$ and we record the identity
\[
\mathcal{E} \psi (x,t)= (2^{k})^n 
e^{i(x\cdot\eta+t \eta^2)} \mathcal{E} \psi_L(2^k (x+2t \gamma), 2^{2k} t) .
\]
If $\psi$ is supported in the set $\Omega$, that is  
\[
\mathcal{E}\psi = \int_{\Omega} e^{i(x\cdot\xi+t \xi^2)} \psi(\xi) d\xi,
\]
then $\psi_L$ is supported in the set $L(\Omega)$. Note that
\begin{align*}
    \|\psi\|_{L^2} = 2^{-\frac{3k}2} \|\psi_{L} \|_{L^2}. 
\end{align*}

We now study the effect of $L$ on the structure of the set of parallelepipeds $\mathcal{P}$. The translation $T_\gamma$ brings a Galilean transformation $(x,t) \mapsto (x+2t\eta,t)$ which leaves the parallelepipeds essentially invariant. To be precise, for a time band interval $[2^{-2k}t_0, 2^{-2k}(t_0+1)], t_0 \in \Z$ the parallelepipeds in it suffer a full translation by a factor of $2^{-2k+1} t_0$, followed by an additional factor of $2^{-2k}(t-t_0)$ whose effect is local in that a parallelepiped $P$ ``leaks" into its neighboring ones. 

The dilation $S_{2^k}$ brings a parabolic rescaling effect $(x,t) \mapsto (2^{k} x, 2^{2k} t)$. Under this transformation the parallelepipeds essentially become thin tubes, the set of which we denote by $\mathcal{T}$, defined as follows. We let $\mathcal{T}$ be the set of parallelepipeds with sides parallel to the coordinate axes and whose sides have length $2^{-k}$ in the direction of $e_3$ and $1$ in all other directions.  Explicitly, a parallelepiped $T \in \mathcal{T}$ is a spacetime set of the form
\[
T= T_z= \{ x \in \R^4:  x_i \in [ z_i , z_i+1], i=1,2,4;
 x_3 \in [2^{-k} z_i , 2^{-k}(z_i+1)] \},
\]
where $z \in \Z^4$ and we write $x=(x_1,x_2,x_3,x_4)$. We intentionally use the letter $T$ for parallelepipeds in $\mathcal{T}$ since their shape is close to that of a tube.

We define the quasi-norm $\ell^p_\mathcal{T} L^1$ in the same way as above, that is for a function $F: \R^4 \rightarrow \C$, we use the notation
\[
\|F\|_{\ell^p_{\mathcal{T}} L^1}: = \left( \sum_{T \in \mathcal{T}} \| F \|_{L^1(T)}^p\right)^\frac1p. 
\]
 
Finally, we study the effect of $L$ on the numerology of the estimate \eqref{triro}. Only the dilation operator affects the estimate; in particular, the $L^1$ norm on the left-hand side of \eqref{triro} scales by a factor of $2^{5k}=2^{3k+2k}$ (in the parabolic rescaling each spatial dimension scales by a factor of $2^k$, and the time variables scales by a factor of $2^{2k}$)
\[
    \|\mathcal{E}\hat{f} \cdot \mathcal{E}\hat{g} \cdot \mathcal{E}\hat{h}\|_{\ell^p_{\mathcal{P}} L^1} \approx 2^{-5k} \|\mathcal{E}\hat{f}_L \cdot \mathcal{E}\hat{g}_L \cdot \mathcal{E}\hat{h}_L\|_{\ell^p_{\mathcal{T}} L^1}.
    \]
Combined with the scaling of the $L^2$ norms under $L$, we see that in order to prove \eqref{triro}, it therefore suffices to prove the estimate 
\begin{align} \label{trirom}
    \norm{\mathcal{E}\hat{f}_L \cdot \mathcal{E}\hat{g}_L \cdot \mathcal{E}\hat{h}_L }_{\ell^p_{\mathcal{T}}L^1} \leq C_\epsilon 2^{-k\epsilon} \cdot 2^{\frac{k}{2}} \norm{f_L}_{L^2} \norm{g_L}_{L^2} \norm{h_L}_{L^2},
\end{align}
where the Fourier supports satisfy \eqref{Ngood}.

This estimate is closer to what we want but the parametrization of the extension operator for the paraboloid \eqref{eq extension operator} is no longer the natural one for $\hat{f}_L$ due to its high frequency support -  a natural parametrization should be close to the intrinsic formulation of the extension operator, see \cite{BeCh} for more details. Technically, the parametrization $\tau(\xi):=|\xi|^2$ of the paraboloid fails the basic requirement of a smooth parametrization
\[
| \partial_{\xi}^\alpha \tau| \les_\alpha 1,
\]
once $|\xi| \gg 1$ and for $|\alpha|=1$. This is necessary in order to invoke the theory in the companion paper \cite{BeCh} and this is why below we adjust the parametrization of the paraboloid accordingly. 

In particular, given the localizations in place, the normal direction to the paraboloid in this set is roughly in the $e_3$ direction. Thus, we rewrite the equation of the paraboloid $\tau=\abs{\xi}^2$ as $\xi_3=\sqrt{\tau-\xi_1^2-\xi_2^2}\eqdef \varphi(\xi_1,\xi_2,\tau)$. Based on this we then write
\begin{align*}
    \mathcal{E}\hat{f}_L=\int e^{i(x\cdot\xi+t\abs{\xi}^2} \widehat{f}_L(\xi)d\xi=\int e^{i((x_1,x_2,t)\cdot (\xi_1,\xi_2,\tau) + x_3\varphi(\xi_1,\xi_2,\tau))} F(\xi_1,\xi_2,\tau) d\xi_1 d\xi_2 d\tau \eqdef \tilde{\mathcal{E}}F,
\end{align*}
where
\begin{align*}
    F(\xi_1,\xi_2,\tau)=\hat{f}_L(\xi_1,\xi_2,\varphi(\xi_1,\xi_2,\tau)) \partial_\tau \varphi(\xi_1,\xi_2,\tau)
\end{align*}
A direct computation shows that $\|F\|_{L^2_{\xi_1,\xi_2,\tau}} \approx 2^{\frac{k}2} \|\hat{f}_L\|_{L^2_\xi}$. Hence, in this new set-up, \eqref{trirom} is equivalent to the estimate
\begin{equation} \label{triromf}
\| \tilde{\mathcal{E}} F \cdot \mathcal{E} G \cdot \mathcal{E} H \|_{\ell^p_{\mathcal{T}}L^1} \leq C_\epsilon 2^{-k\epsilon} \|F\|_{L^2} \|G\|_{L^2} \|H\|_{L^2}. 
\end{equation}
We note that $\tilde{\mathcal{E}}$ is an extension operator associated to the same underlying hypersurface, the paraboloid, but it is just rewritten in more natural coordinates in the sense described above.

The proof of \eqref{triromf} follows from the following corollary from \cite{BeCh}, and this concludes the proof of Theorem \ref{ltri1}.

\begin{lemma}[{\cite[Corollary 1.5]{BeCh}}]
Let $p>\frac{14}{15}$. Let $\mathbb{P}=\{(\xi,|\xi|^2): \xi \in \R^3\}$ denote the paraboloid in $\R^4$ and let $S_i=\mathbb{P}\cap B_1(x_i)$ where $\abs{x_1},\abs{x_2}\leq 8$ and $\abs{x_3}\approx 2^{-k}$ such that they satisfy \eqref{ndeg} with $\nu\approx 1$, that is
\begin{align*}
    \abs{N(\zeta_1)\wedge N(\zeta_2) \wedge N(\zeta_3)} \gtrsim 1\quad\forall \zeta_i\in S_i.
\end{align*}
Let $\mathcal{E}_i$ denote the extension operator associated to each $S_i$. Then for all $\epsilon>0$ there exists $C_\epsilon>0$ such that
\begin{align}
    \| \prod_{i=1}^3  \tilde{\mathcal{E}}_i f_i\|_{\ell^p_\mathcal{P} L^{1}} \leq C_\epsilon 2^{-k\epsilon} \prod_{i=1}^3 \|f_i\|_{L^2}.
\end{align}
\end{lemma}

We remark that the proof of this result uses polynomial partitioning and other ideas from Fourier restriction theory, based on \cite{Be-mrec}, which is very different from this current paper. For further context and discussion about these ideas, see for example \cite{Gu-I, Gu-II, Be-mrec, BeCh} and the references therein.

\appendix \section{Transference to \texorpdfstring{$U^2_\Delta$}{U2}}

We establish how the trilinear estimate \eqref{triro} for free solutions can be transferred to functions in $U^2_\Delta$, which proves \eqref{me1}. We closely follow the argument of Candy in \cite[Theorem 7]{Ca-note}. Note that we cannot directly quote that argument in our case due to the usage of the quasi-norm $\ell^p_{\mathcal{P}} L^1$ with $p<1$ but the same strategy can be easily adapted to this setting. For clarity and convenience, we will omit the $C_\epsilon 2^{-k\epsilon}$ term as this constant is easily seen to remain essentially unaffected throughout the argument.

We start with an atom $u_1 = \sum_{I \in \mathcal{I}} e^{-it\Delta} f_I$ in $U^2_\Delta$, while letting $u_2= e^{-it\Delta} g, u_3= e^{-it\Delta} h$ be free solutions. Let $(\epsilon_I)_{I \in \mathcal{I}}$ be a family of independent, identically distributed Rademacher random variables, i.e. $\epsilon_I = 1$
with probability $\frac12$, $\epsilon_I = -1$ with probability $\frac12$. 
Since the intervals $I$ are disjoint, by invoking Khintchine's inequality we have
\[
|u_1| \les \left( \sum_I |e^{-it\Delta} f_I|^2 \right)^\frac12 \approx  \left(\mathbb{E} |\sum_I \epsilon_I e^{-it\Delta} f_I|^p\right)^{\frac{1}{p}}.
\]
Then we have
\[
\begin{split}
& \|u_1 u_2 u_3\|^p_{\ell^p_{\mathcal{P}} L^1}  = \sum_P   \|u_1 u_2 u_3\|^p_{L^1(P)} 
  \les \sum_P   \|(\mathbb{E}|\sum_I \epsilon_I e^{-it\Delta} f_I|^p)^\frac1p u_2 u_3\|^p_{L^1(P)} \\
 \les & \sum_P   \|(\mathbb{E}|\sum_I \epsilon_I e^{-it\Delta} f_I|^p |u_2|^p |u_3|^p)^\frac1p \|^p_{L^1(P)}
= \sum_P   \|\mathbb{E}|\sum_I \epsilon_I e^{-it\Delta} f_I|^p |u_2|^p |u_3|^p \|_{L^\frac1p(P)}\\
 \les & \sum_P   \mathbb{E} \||\sum_I \epsilon_I e^{-it\Delta} f_I|^p |u_2|^p |u_3|^p \|_{L^\frac1p(P)}
= \mathbb{E} \sum_P \||\sum_I \epsilon_I e^{-it\Delta} f_I|^p |u_2|^p |u_3|^p \|_{L^\frac1p(P)}\\
 = & \mathbb{E} \sum_P \||\sum_I \epsilon_I e^{-it\Delta} f_I| |u_2| |u_3| \|_{L^1(P)}^p  \les \mathbb{E} \|\sum_I \epsilon_I e^{-it\Delta} f_I\|_{L^2}^p \|g\|_{L^2}^p \|h\|_{L^2}^p \\
  = &\mathbb{E} \|\sum_I \epsilon_I f_I\|_{L^2}^p \|g\|_{L^2}^p \|h\|_{L^2}^p.
\end{split}
\]
In the above we used that $\frac1p \geq 1$ in order to invoke Minkowski's inequality to pass to the the first inequality on the third line. As a consequence,
\[
\|u_1 u_2 u_3\|_{\ell^p_{\mathcal{P}} L^1} \les (\mathbb{E} \|\sum_I \epsilon_I f_I\|_{L^2}^p)^\frac1p \|u_2\|_{L^2} \|u_3\|_{L^2}.
\]
Another application of Khintchine's inequality shows that 
\[
(\mathbb{E} \|\sum_I \epsilon_I f_I\|_{L^2}^p)^\frac1p \les 
(\mathbb{E} \|\sum_I \epsilon_I f_I\|_{L^2}^2)^\frac12= \left( 
\sum_I \|f_I\|_{L^2}^2 \right)^\frac12 = \|u_1\|_{U^2_\Delta}, 
\]
from which we conclude with 
\begin{equation} \label{rep1}
\|u_1 u_2 u_3\|_{\ell^p_{\mathcal{P}} L^1} \les \|u_1\|_{U^2_\Delta} \|g\|_{L^2} \|h\|_{L^2}.
\end{equation}
In this inequality we generalized the first free wave to an element in $U^2_\Delta$. Then, just as in \cite{Ca-note}, a repetition of the above argument with $u_2$ playing the role of $u_1$ above, that is being the atom in $U^2_\Delta$, and using \eqref{rep1}, generalizes the second free wave to an element in $U^2_\Delta$. One last repetition generates the full transference as stated in \eqref{me1}.

We note that other arguments for transferring multilinear estimates to $U^2_\Delta$ have been used in the literature, and any of them could have potentially been used here as an alternative proof. First, one could try to generalize the trilinear restriction estimate to vector-valued functions and then run the transference argument for $U^2_\Delta$ atoms, as in \cite{Ca}. Second, using the wave packet theory for $U^2_\Delta$ developed in \cite{CaHe-t}, one could try to prove the trilinear restriction estimate in $U^2_\Delta$ directly. We do not pursue these ideas here, and we simply include them for the sake of completeness.

\bibliographystyle{amsplain} \bibliography{QNLS-refs}

@article {Bea,
    AUTHOR = {Beals, Michael},
     TITLE = {Self-spreading and strength of singularities for solutions to
              semilinear wave equations},
   JOURNAL = {Ann. of Math. (2)},
  FJOURNAL = {Annals of Mathematics. Second Series},
    VOLUME = {118},
      YEAR = {1983},
    NUMBER = {1},
     PAGES = {187--214},
      ISSN = {0003-486X,1939-8980},
   MRCLASS = {35L70 (58G17)},
  MRNUMBER = {707166},
MRREVIEWER = {K.\ A.\ Yagdjian},
       DOI = {10.2307/2006959},
       URL = {https://doi.org/10.2307/2006959},
}

@article {Be-mrec,
    AUTHOR = {Bejenaru, Ioan},
     TITLE = {The almost optimal multilinear restriction estimate for
              hypersurfaces with curvature: the case of {$n-1$}
              hypersurfaces in {$\Bbb R^n$}},
   JOURNAL = {Int. Math. Res. Not. IMRN},
  FJOURNAL = {International Mathematics Research Notices. IMRN},
      YEAR = {2022},
    NUMBER = {20},
     PAGES = {16363--16404},
      ISSN = {1073-7928,1687-0247},
   MRCLASS = {53C21},
  MRNUMBER = {4498177},
MRREVIEWER = {Jeffrey\ S.\ Case},
       DOI = {10.1093/imrn/rnab208},
       URL = {https://doi.org/10.1093/imrn/rnab208},
}

@article {BeDa,
    AUTHOR = {Bejenaru, Ioan and De Silva, Daniela},
     TITLE = {Low regularity solutions for a 2{D} quadratic nonlinear
              {S}chr\"odinger equation},
   JOURNAL = {Trans. Amer. Math. Soc.},
  FJOURNAL = {Transactions of the American Mathematical Society},
    VOLUME = {360},
      YEAR = {2008},
    NUMBER = {11},
     PAGES = {5805--5830},
      ISSN = {0002-9947,1088-6850},
   MRCLASS = {35Q55 (35A07 35B30 35B45)},
  MRNUMBER = {2425692},
MRREVIEWER = {Thomas\ Duyckaerts},
       DOI = {10.1090/S0002-9947-08-04415-2},
       URL = {https://doi.org/10.1090/S0002-9947-08-04415-2},
}

@article {BeHe,
    AUTHOR = {Bejenaru, Ioan and Herr, Sebastian},
     TITLE = {The cubic {D}irac equation: small initial data in {$H^{\frac
              12}(\Bbb R^2)$}},
   JOURNAL = {Comm. Math. Phys.},
  FJOURNAL = {Communications in Mathematical Physics},
    VOLUME = {343},
      YEAR = {2016},
    NUMBER = {2},
     PAGES = {515--562},
      ISSN = {0010-3616,1432-0916},
   MRCLASS = {35Q41 (35B30 35L10 35P25)},
  MRNUMBER = {3477346},
MRREVIEWER = {Marco\ Falconi},
       DOI = {10.1007/s00220-015-2508-4},
       URL = {https://doi.org/10.1007/s00220-015-2508-4},
}

@article {BeHeTa,
    AUTHOR = {Bejenaru, Ioan and Herr, Sebastian and Tataru, Daniel},
     TITLE = {A convolution estimate for two-dimensional hypersurfaces},
   JOURNAL = {Rev. Mat. Iberoam.},
  FJOURNAL = {Revista Matem\'atica Iberoamericana},
    VOLUME = {26},
      YEAR = {2010},
    NUMBER = {2},
     PAGES = {707--728},
      ISSN = {0213-2230,2235-0616},
   MRCLASS = {42B35 (35Q55 42B37 47B38)},
  MRNUMBER = {2677013},
MRREVIEWER = {Michael\ J.\ Goldberg},
       DOI = {10.4171/RMI/615},
       URL = {https://doi.org/10.4171/RMI/615},
}

@article{Be-rs,
    author = {Bejenaru, Ioan},
    title = {A new resolution space for nonlinear {S}chr\"odinger equations and applications},
    journal = {},
    year = {2026},
    note = {Preprint available at \url{https://arxiv.org/abs/2606.13927}},
}

@article{BeCh,
    author = {Bejenaru, Ioan and Chan, Tiklung},
    title = {The $(n-1)$-linear restriction estimate in $\mathbb{R}^n$ with degenerate curvature},
    journal = {},
    year = {2026},
    note = {Preprint},
}

@article {BIKT-1,
  AUTHOR =	 {Bejenaru, Ioan and Ionescu, Alexandru D. and Kenig, Carlos E. and
                  Tataru, Daniel},
  TITLE =	 {Global {S}chr\"odinger maps in dimensions {$d\geq
                  2$}: small data in the critical {S}obolev spaces},
  JOURNAL =	 {Ann. of Math. (2)},
  FJOURNAL =	 {Annals of Mathematics. Second Series},
  VOLUME =	 173,
  YEAR =	 2011,
  NUMBER =	 3,
  PAGES =	 {1443--1506},
  ISSN =	 {0003-486X},
  CODEN =	 {ANMAAH},
  MRCLASS =	 {58J35 (35B65 35K45 46E35)},
  MRNUMBER =	 {2800718 (2012g:58048)},
  MRREVIEWER =	 {Atanas G. Stefanov},
  DOI =		 {10.4007/annals.2011.173.3.5},
  URL =		 {http://dx.doi.org/10.4007/annals.2011.173.3.5},
}

@article {BeTa-2,
    AUTHOR = {Bejenaru, Ioan and Tataru, Daniel},
     TITLE = {Large data local solutions for the derivative {NLS} equation},
   JOURNAL = {J. Eur. Math. Soc. (JEMS)},
  FJOURNAL = {Journal of the European Mathematical Society (JEMS)},
    VOLUME = {10},
      YEAR = {2008},
    NUMBER = {4},
     PAGES = {957--985},
      ISSN = {1435-9855,1435-9863},
   MRCLASS = {35Q55 (35B30)},
  MRNUMBER = {2443925},
MRREVIEWER = {Peter\ E.\ Zhidkov},
       DOI = {10.4171/JEMS/136},
       URL = {https://doi.org/10.4171/JEMS/136},
}

@article {BeCaTa,
    AUTHOR = {Bennett, Jonathan and Carbery, Anthony and Tao, Terence},
     TITLE = {On the multilinear restriction and {K}akeya conjectures},
   JOURNAL = {Acta Math.},
  FJOURNAL = {Acta Mathematica},
    VOLUME = {196},
      YEAR = {2006},
    NUMBER = {2},
     PAGES = {261--302},
      ISSN = {0001-5962},
     CODEN = {ACMAA8},
   MRCLASS = {42B20 (44A15 45P05 46E35)},
  MRNUMBER = {2275834 (2007h:42019)},
MRREVIEWER = {Loukas Grafakos},
       DOI = {10.1007/s11511-006-0006-4},
       URL = {http://dx.doi.org/10.1007/s11511-006-0006-4},
}

@article {BeCaWr,
    AUTHOR = {Bennett, Jonathan and Carbery, Anthony and Wright, James},
     TITLE = {A non-linear generalisation of the {L}oomis-{W}hitney
              inequality and applications},
   JOURNAL = {Math. Res. Lett.},
  FJOURNAL = {Mathematical Research Letters},
    VOLUME = {12},
      YEAR = {2005},
    NUMBER = {4},
     PAGES = {443--457},
      ISSN = {1073-2780},
   MRCLASS = {26D15 (42B10)},
  MRNUMBER = {2155223},
MRREVIEWER = {D.\ C.\ Russell},
       DOI = {10.4310/MRL.2005.v12.n4.a1},
       URL = {https://doi.org/10.4310/MRL.2005.v12.n4.a1},
}

@article {Bo-1,
    AUTHOR = {Bourgain, J.},
     TITLE = {Fourier transform restriction phenomena for certain lattice
              subsets and applications to nonlinear evolution equations.
              {I}. {S}chr\"odinger equations},
   JOURNAL = {Geom. Funct. Anal.},
  FJOURNAL = {Geometric and Functional Analysis},
    VOLUME = {3},
      YEAR = {1993},
    NUMBER = {2},
     PAGES = {107--156},
      ISSN = {1016-443X,1420-8970},
   MRCLASS = {35Q55 (11L07 35B10)},
  MRNUMBER = {1209299},
MRREVIEWER = {Yun\ Mei\ Chen},
       DOI = {10.1007/BF01896020},
       URL = {https://doi.org/10.1007/BF01896020},
}

@article {Bo-2,
    AUTHOR = {Bourgain, J.},
     TITLE = {Fourier transform restriction phenomena for certain lattice
              subsets and applications to nonlinear evolution equations.
              {II}. {T}he {K}d{V}-equation},
   JOURNAL = {Geom. Funct. Anal.},
  FJOURNAL = {Geometric and Functional Analysis},
    VOLUME = {3},
      YEAR = {1993},
    NUMBER = {3},
     PAGES = {209--262},
      ISSN = {1016-443X,1420-8970},
   MRCLASS = {35Q55},
  MRNUMBER = {1215780},
MRREVIEWER = {Yun\ Mei\ Chen},
       DOI = {10.1007/BF01895688},
       URL = {https://doi.org/10.1007/BF01895688},
}

@article {Ca,
    AUTHOR = {Candy, Timothy},
     TITLE = {Multi-scale bilinear restriction estimates for general phases},
   JOURNAL = {Math. Ann.},
  FJOURNAL = {Mathematische Annalen},
    VOLUME = {375},
      YEAR = {2019},
    NUMBER = {1-2},
     PAGES = {777--843},
      ISSN = {0025-5831,1432-1807},
   MRCLASS = {42B37 (35L10 35Q41)},
  MRNUMBER = {4000257},
MRREVIEWER = {Piero\ Antonio\ D'Ancona},
       DOI = {10.1007/s00208-019-01841-4},
       URL = {https://doi.org/10.1007/s00208-019-01841-4},
}

@incollection {Ca-note,
    AUTHOR = {Candy, Timothy},
     TITLE = {A note on bilinear wave-{S}chr\"odinger interactions},
 BOOKTITLE = {2019--20 {MATRIX} annals},
    SERIES = {MATRIX Book Ser.},
    VOLUME = {4},
     PAGES = {537--549},
 PUBLISHER = {Springer, Cham},
      YEAR = {[2021] \copyright 2021},
      ISBN = {978-3-030-62496-5; 978-3-030-62497-2},
   MRCLASS = {35Q55},
  MRNUMBER = {4294792},
       DOI = {10.1007/978-3-030-62497-2\_33},
       URL = {https://doi.org/10.1007/978-3-030-62497-2_33},
}

@article {CaHe-t,
    AUTHOR = {Candy, Timothy and Herr, Sebastian},
     TITLE = {Transference of bilinear restriction estimates to quadratic
              variation norms and the {D}irac-{K}lein-{G}ordon system},
   JOURNAL = {Anal. PDE},
  FJOURNAL = {Analysis \& PDE},
    VOLUME = {11},
      YEAR = {2018},
    NUMBER = {5},
     PAGES = {1171--1240},
      ISSN = {2157-5045,1948-206X},
   MRCLASS = {42B37 (35Q41 42B10 42B20 81Q05)},
  MRNUMBER = {3785603},
MRREVIEWER = {Ahmed\ Lesfari},
       DOI = {10.2140/apde.2018.11.1171},
       URL = {https://doi.org/10.2140/apde.2018.11.1171},
}

@article {CaHe,
    AUTHOR = {Candy, Timothy and Herr, Sebastian},
     TITLE = {On the division problem for the wave maps equation},
   JOURNAL = {Ann. PDE},
  FJOURNAL = {Annals of PDE. Journal Dedicated to the Analysis of Problems
              from Physical Sciences},
    VOLUME = {4},
      YEAR = {2018},
    NUMBER = {2},
     PAGES = {Paper No. 17, 61},
      ISSN = {2524-5317,2199-2576},
   MRCLASS = {35L51 (35B30 35L52 35L72 35P25)},
  MRNUMBER = {3883339},
MRREVIEWER = {Rodica\ Luca},
       DOI = {10.1007/s40818-018-0054-z},
       URL = {https://doi.org/10.1007/s40818-018-0054-z},
}

@book {Caz-book,
    AUTHOR = {Cazenave, Thierry},
     TITLE = {Semilinear {S}chr\"odinger equations},
    SERIES = {Courant Lecture Notes in Mathematics},
    VOLUME = {10},
 PUBLISHER = {New York University, Courant Institute of Mathematical
              Sciences, New York; American Mathematical Society, Providence,
              RI},
      YEAR = {2003},
     PAGES = {xiv+323},
      ISBN = {0-8218-3399-5},
   MRCLASS = {35Q55 (35-01 35J10 35Q40)},
  MRNUMBER = {2002047},
MRREVIEWER = {Woodford\ W.\ Zachary},
       DOI = {10.1090/cln/010},
       URL = {https://doi.org/10.1090/cln/010},
}

@article {CaWe,
    AUTHOR = {Cazenave, Thierry and Weissler, Fred B.},
     TITLE = {The {C}auchy problem for the critical nonlinear
              {S}chr\"odinger equation in {$H^s$}},
   JOURNAL = {Nonlinear Anal.},
  FJOURNAL = {Nonlinear Analysis. Theory, Methods \& Applications. An
              International Multidisciplinary Journal},
    VOLUME = {14},
      YEAR = {1990},
    NUMBER = {10},
     PAGES = {807--836},
      ISSN = {0362-546X,1873-5215},
   MRCLASS = {35Q55 (45K05)},
  MRNUMBER = {1055532},
MRREVIEWER = {Victor\ V.\ Zharinov},
       DOI = {10.1016/0362-546X(90)90023-A},
       URL = {https://doi.org/10.1016/0362-546X(90)90023-A},
}

@article {GeMaSh-1,
    AUTHOR = {Germain, Pierre and Masmoudi, Nader and Shatah, Jalal},
     TITLE = {Global solutions for 3{D} quadratic {S}chr\"odinger equations},
   JOURNAL = {Int. Math. Res. Not. IMRN},
  FJOURNAL = {International Mathematics Research Notices. IMRN},
      YEAR = {2009},
    NUMBER = {3},
     PAGES = {414--432},
      ISSN = {1073-7928,1687-0247},
   MRCLASS = {35Q55},
  MRNUMBER = {2482120},
MRREVIEWER = {Pascal\ B\'egout},
       DOI = {10.1093/imrn/rnn135},
       URL = {https://doi.org/10.1093/imrn/rnn135},
}

@article {Gu-main,
    AUTHOR = {Guth, Larry},
     TITLE = {The endpoint case of the {B}ennett-{C}arbery-{T}ao multilinear
              {K}akeya conjecture},
   JOURNAL = {Acta Math.},
  FJOURNAL = {Acta Mathematica},
    VOLUME = {205},
      YEAR = {2010},
    NUMBER = {2},
     PAGES = {263--286},
      ISSN = {0001-5962},
     CODEN = {ACMAA8},
   MRCLASS = {42B15 (28A75 52A38)},
  MRNUMBER = {2746348 (2012c:42027)},
MRREVIEWER = {Dmitry Ryabogin},
       DOI = {10.1007/s11511-010-0055-6},
       URL = {http://dx.doi.org/10.1007/s11511-010-0055-6},
}

@article {Gu-I,
    AUTHOR = {Guth, Larry},
     TITLE = {A restriction estimate using polynomial partitioning},
   JOURNAL = {J. Amer. Math. Soc.},
  FJOURNAL = {Journal of the American Mathematical Society},
    VOLUME = {29},
      YEAR = {2016},
    NUMBER = {2},
     PAGES = {371--413},
      ISSN = {0894-0347},
   MRCLASS = {42B20},
  MRNUMBER = {3454378},
       DOI = {10.1090/jams827},
       URL = {http://dx.doi.org/10.1090/jams827},
}

@article {Gu-II,
    AUTHOR = {Guth, Larry},
     TITLE = {Restriction estimates using polynomial partitioning {II}},
   JOURNAL = {Acta Math.},
  FJOURNAL = {Acta Mathematica},
    VOLUME = {221},
      YEAR = {2018},
    NUMBER = {1},
     PAGES = {81--142},
      ISSN = {0001-5962},
   MRCLASS = {42B20},
  MRNUMBER = {3877019},
       DOI = {10.4310/ACTA.2018.v221.n1.a3},
       URL = {https://doi.org/10.4310/ACTA.2018.v221.n1.a3},
}

@article {HaMiNa,
    AUTHOR = {Hayashi, Nakao and Mizumachi, Tetsu and Naumkin, Pavel I.},
     TITLE = {Time decay of small solutions to quadratic nonlinear
              {S}chr\"odinger equations in 3{D}},
   JOURNAL = {Differential Integral Equations},
  FJOURNAL = {Differential and Integral Equations. An International Journal
              for Theory \& Applications},
    VOLUME = {16},
      YEAR = {2003},
    NUMBER = {2},
     PAGES = {159--179},
      ISSN = {0893-4983},
   MRCLASS = {35Q55 (35B40 81Q05)},
  MRNUMBER = {1947090},
MRREVIEWER = {Arthur\ H.\ Vartanian},
}

@article {HaNa,
    AUTHOR = {Hayashi, Nakao and Naumkin, Pavel I.},
     TITLE = {On the quadratic nonlinear {S}chr\"odinger equation in three
              space dimensions},
   JOURNAL = {Internat. Math. Res. Notices},
  FJOURNAL = {International Mathematics Research Notices},
      YEAR = {2000},
    NUMBER = {3},
     PAGES = {115--132},
      ISSN = {1073-7928,1687-0247},
   MRCLASS = {35Q55 (35K15)},
  MRNUMBER = {1741610},
MRREVIEWER = {Grzegorz\ Karch},
       DOI = {10.1155/S1073792800000088},
       URL = {https://doi.org/10.1155/S1073792800000088},
}

@article {HHK,
    AUTHOR = {Hadac, Martin and Herr, Sebastian and Koch, Herbert},
     TITLE = {Well-posedness and scattering for the {KP}-{II} equation in a
              critical space},
   JOURNAL = {Ann. Inst. H. Poincar\'e Anal. Non Lin\'eaire},
  FJOURNAL = {Annales de l'Institut Henri Poincar\'e. Analyse Non
              Lin\'eaire},
    VOLUME = {26},
      YEAR = {2009},
    NUMBER = {3},
     PAGES = {917--941},
      ISSN = {0294-1449},
   MRCLASS = {35Q53 (35B30 35P25)},
  MRNUMBER = {2526409 (2010d:35301)},
MRREVIEWER = {Ming Mei},
       DOI = {10.1016/j.anihpc.2008.04.002},
       URL = {http://dx.doi.org/10.1016/j.anihpc.2008.04.002},
}

@article {HeTaTz,
    AUTHOR = {Herr, Sebastian and Tataru, Daniel and Tzvetkov, Nikolay},
     TITLE = {Global well-posedness of the energy-critical nonlinear
              {S}chr\"odinger equation with small initial data in {$H^1(\Bbb
              T^3)$}},
   JOURNAL = {Duke Math. J.},
  FJOURNAL = {Duke Mathematical Journal},
    VOLUME = {159},
      YEAR = {2011},
    NUMBER = {2},
     PAGES = {329--349},
      ISSN = {0012-7094,1547-7398},
   MRCLASS = {35Q55 (35B30)},
  MRNUMBER = {2824485},
MRREVIEWER = {Matthew\ D.\ Blair},
       DOI = {10.1215/00127094-1415889},
       URL = {https://doi.org/10.1215/00127094-1415889},
}

@article {Ka,
    AUTHOR = {Kawahara, Yuichiro},
     TITLE = {Global existence and asymptotic behavior of small solutions to
              nonlinear {S}chr\"odinger equations in 3{D}},
   JOURNAL = {Differential Integral Equations},
  FJOURNAL = {Differential and Integral Equations. An International Journal
              for Theory \& Applications},
    VOLUME = {18},
      YEAR = {2005},
    NUMBER = {2},
     PAGES = {169--194},
      ISSN = {0893-4983},
   MRCLASS = {35Q55 (35B40)},
  MRNUMBER = {2106101},
MRREVIEWER = {Yvan\ Martel},
}

@article {KeGuVe,
    AUTHOR = {Kenig, Carlos E. and Ponce, Gustavo and Vega, Luis},
     TITLE = {Quadratic forms for the {$1$}-{D} semilinear {S}chr\"odinger
              equation},
   JOURNAL = {Trans. Amer. Math. Soc.},
  FJOURNAL = {Transactions of the American Mathematical Society},
    VOLUME = {348},
      YEAR = {1996},
    NUMBER = {8},
     PAGES = {3323--3353},
      ISSN = {0002-9947,1088-6850},
   MRCLASS = {35Q55 (35B45 47F05)},
  MRNUMBER = {1357398},
MRREVIEWER = {Peter\ Stollmann},
       DOI = {10.1090/S0002-9947-96-01645-5},
       URL = {https://doi.org/10.1090/S0002-9947-96-01645-5},
}

@article {KlMa,
    AUTHOR = {Klainerman, S. and Machedon, M.},
     TITLE = {Space-time estimates for null forms and the local existence
              theorem},
   JOURNAL = {Comm. Pure Appl. Math.},
  FJOURNAL = {Communications on Pure and Applied Mathematics},
    VOLUME = {46},
      YEAR = {1993},
    NUMBER = {9},
     PAGES = {1221--1268},
      ISSN = {0010-3640,1097-0312},
   MRCLASS = {35L05 (35B45 35L15 35L70)},
  MRNUMBER = {1231427},
MRREVIEWER = {Vladimir\ Georgiev},
       DOI = {10.1002/cpa.3160460902},
       URL = {https://doi.org/10.1002/cpa.3160460902},
}

@article {KT1,
    AUTHOR = {Koch, Herbert and Tataru, Daniel},
     TITLE = {A priori bounds for the 1{D} cubic {NLS} in negative {S}obolev
              spaces},
   JOURNAL = {Int. Math. Res. Not. IMRN},
  FJOURNAL = {International Mathematics Research Notices. IMRN},
      YEAR = {2007},
    NUMBER = {16},
     PAGES = {Art. ID rnm053, 36},
      ISSN = {1073-7928,1687-0247},
   MRCLASS = {35Q53 (35B45 37K05 37K10)},
  MRNUMBER = {2353092},
       DOI = {10.1093/imrn/rnm053},
       URL = {https://doi.org/10.1093/imrn/rnm053},
}

@article {KT,
    AUTHOR = {Koch, Herbert and Tataru, Daniel},
     TITLE = {Dispersive estimates for principally normal pseudodifferential
              operators},
   JOURNAL = {Comm. Pure Appl. Math.},
  FJOURNAL = {Communications on Pure and Applied Mathematics},
    VOLUME = {58},
      YEAR = {2005},
    NUMBER = {2},
     PAGES = {217--284},
      ISSN = {0010-3640},
     CODEN = {CPAMA},
   MRCLASS = {35S05 (35B45 47G30)},
  MRNUMBER = {2094851 (2005m:35323)},
MRREVIEWER = {Yu. V. Egorov},
       DOI = {10.1002/cpa.20067},
       URL = {http://dx.doi.org/10.1002/cpa.20067},
}

@article {KT2,
    AUTHOR = {Koch, Herbert and Tataru, Daniel},
     TITLE = {Conserved energies for the cubic nonlinear {S}chr\"odinger
              equation in one dimension},
   JOURNAL = {Duke Math. J.},
  FJOURNAL = {Duke Mathematical Journal},
    VOLUME = {167},
      YEAR = {2018},
    NUMBER = {17},
     PAGES = {3207--3313},
      ISSN = {0012-7094,1547-7398},
   MRCLASS = {35Q55 (35Q53 37K05 37K10)},
  MRNUMBER = {3874652},
MRREVIEWER = {Thierry\ Cazenave},
       DOI = {10.1215/00127094-2018-0033},
       URL = {https://doi.org/10.1215/00127094-2018-0033},
}

@book {KTV,
    AUTHOR = {Koch, Herbert and Tataru, Daniel and Vi{\c s}an, Monica},
     TITLE = {Dispersive equations and nonlinear waves},
    SERIES = {Oberwolfach Seminars},
    VOLUME = {45},
      NOTE = {Generalized Korteweg-de Vries, nonlinear Schr\"odinger, wave
              and Schr\"odinger maps},
 PUBLISHER = {Birkh\"auser/Springer, Basel},
      YEAR = {2014},
     PAGES = {xii+312},
      ISBN = {978-3-0348-0735-7; 978-3-0348-0736-4},
   MRCLASS = {35-02 (35L70 35Q51 35Q53 35Q55)},
  MRNUMBER = {3618884},
}

@article {MuToTa,
    AUTHOR = {Muramatu, Tosinobu and Taoka, Shifu},
     TITLE = {The initial value problem for the 1-{D} semilinear
              {S}chr\"odinger equation in {B}esov spaces},
   JOURNAL = {J. Math. Soc. Japan},
  FJOURNAL = {Journal of the Mathematical Society of Japan},
    VOLUME = {56},
      YEAR = {2004},
    NUMBER = {3},
     PAGES = {853--888},
      ISSN = {0025-5645,1881-1167},
   MRCLASS = {35Q55 (35G25 46E35)},
  MRNUMBER = {2071676},
MRREVIEWER = {Giulio\ Soliani},
       DOI = {10.2969/jmsj/1191334089},
       URL = {https://doi.org/10.2969/jmsj/1191334089},
}

@incollection {NaTaTs,
    AUTHOR = {Nakanishi, Kenji and Takaoka, Hideo and Tsutsumi, Yoshio},
     TITLE = {Counterexamples to bilinear estimates related with the {K}d{V}
              equation and the nonlinear {S}chr\"odinger equation},
      NOTE = {IMS Conference on Differential Equations from Mechanics (Hong
              Kong, 1999)},
   JOURNAL = {Methods Appl. Anal.},
  FJOURNAL = {Methods and Applications of Analysis},
    VOLUME = {8},
      YEAR = {2001},
    NUMBER = {4},
     PAGES = {569--578},
      ISSN = {1073-2772,1945-0001},
   MRCLASS = {35Q53 (35B45 35Q55)},
  MRNUMBER = {1944182},
MRREVIEWER = {Zoran\ Gruji\'c},
       DOI = {10.4310/MAA.2001.v8.n4.a7},
       URL = {https://doi.org/10.4310/MAA.2001.v8.n4.a7},
}

@article {Ta-2,
  AUTHOR =	 {Tataru, Daniel},
  TITLE =	 {On global existence and scattering for the wave maps
                  equation},
  JOURNAL =	 {Amer. J. Math.},
  FJOURNAL =	 {American Journal of Mathematics},
  VOLUME =	 123,
  YEAR =	 2001,
  NUMBER =	 1,
  PAGES =	 {37--77},
  ISSN =	 {0002-9327},
  CODEN =	 {AJMAAN},
  MRCLASS =	 {58J45 (35L70 35P25)},
  MRNUMBER =	 {1827277 (2002c:58045)},
  MRREVIEWER =	 {Kenji Nakanishi},
  URL =
                  {http://muse.jhu.edu/journals/american_journal_of_mathematics/v123/123.1tataru.pdf},
}

@book {Tao-book,
    AUTHOR = {Tao, Terence},
     TITLE = {Nonlinear dispersive equations},
    SERIES = {CBMS Regional Conference Series in Mathematics},
    VOLUME = {106},
      NOTE = {Local and global analysis},
 PUBLISHER = {Conference Board of the Mathematical Sciences, Washington, DC;
              by the American Mathematical Society, Providence, RI},
      YEAR = {2006},
     PAGES = {xvi+373},
      ISBN = {0-8218-4143-2},
   MRCLASS = {35Q53 (35B35 35P25 35Q55 37K10)},
  MRNUMBER = {2233925},
MRREVIEWER = {Sebastian\ Herr},
       DOI = {10.1090/cbms/106},
       URL = {https://doi.org/10.1090/cbms/106},
}

@article{guo_global_2025,
	title = {Global solutions to 3{D} quadratic nonlinear {S}chrödinger-type equation},
	volume = {13},
	issn = {2050-5094},
	doi = {10.1017/fms.2025.10132},
	pages = {e183},
	journal = {Forum of Mathematics, Sigma},
	author = {Guo, Zihua and Liu, Naijia and Song, Liang},
	date = {2025},
}

@article{tao_multilinear_2001,
	title = {Multilinear weighted convolution of {L} 2 functions, and applications to nonlinear dispersive equations},
	volume = {123},
	issn = {1080-6377},
	doi = {10.1353/ajm.2001.0035},
	pages = {839--908},
	number = {5},
	journal = {American Journal of Mathematics},
	author = {Tao, Terence},
	date = {2001-10},
}

@article{colliander_bilinear_2001,
	title = {Bilinear {E}stimates and {A}pplications to 2{D} {NLS}},
	volume = {353},
	issn = {0002-9947},
	pages = {3307--3325},
	number = {8},
	journal = {Transactions of the American Mathematical Society},
	publisher = {American Mathematical Society},
	author = {Colliander, J. E. and Delort, J.-M. and Kenig, C. E. and Staffilani, G.},
	date = {2001},
}

@article{bejenaru_sharp_2006,
	title = {Sharp well-posedness and ill-posedness results for a quadratic non-linear {S}chrödinger equation},
	volume = {233},
	issn = {0022-1236},
	doi = {10.1016/j.jfa.2005.08.004},
	pages = {228--259},
	number = {1},
	journal = {Journal of Functional Analysis},
	author = {Bejenaru, Ioan and Tao, Terence},
	date = {2006-04-01},
}

@article{ginibre_almost_1995,
	title = {Almost global existence of small solutions to quadratic nonlinear {S}chrödinger equations in three space dimensions},
	volume = {219},
	issn = {1432-1823},
	doi = {10.1007/BF02572354},
	pages = {119--140},
	number = {1},
	journaltitle = {Mathematische Zeitschrift},
	author = {Ginibre, J. and Hayashi, N.},
	date = {1995-05-01},
}

@article{tao_bilinear_1998,
	title = {A Bilinear Approach to the Restriction and Kakeya Conjectures},
	volume = {11},
	issn = {0894-0347},
	url = {https://www.jstor.org/stable/2646145},
	pages = {967--1000},
	number = {4},
	journaltitle = {Journal of the American Mathematical Society},
	publisher = {American Mathematical Society},
	author = {Tao, Terence and Vargas, Ana and Vega, Luis},
	urldate = {2025-06-22},
	date = {1998},
}

@article {Tsu,
    AUTHOR = {Tsutsumi, Yoshio},
     TITLE = {{$L^2$}-solutions for nonlinear {S}chr\"odinger equations and
              nonlinear groups},
   JOURNAL = {Funkcial. Ekvac.},
  FJOURNAL = {Funkcialaj Ekvacioj. Serio Internacia},
    VOLUME = {30},
      YEAR = {1987},
    NUMBER = {1},
     PAGES = {115--125},
      ISSN = {0532-8721},
   MRCLASS = {35Q20 (35J10)},
  MRNUMBER = {915266},
MRREVIEWER = {Woodford\ W.\ Zachary},
       DOI = {10.24546/0100499535},
       URL = {https://doi.org/10.24546/0100499535},
}

@article{zhang_endpoint_2018,
	title = {The endpoint perturbed Brascamp–Lieb inequalities with examples},
	volume = {11},
	pages = {555--581},
	number = {3},
	journaltitle = {Analysis \& {PDE}},
	author = {Zhang, Ruixiang},
	date = {2018-01},
}

\end{document}